\documentclass[11pt]{article}

\usepackage{amsmath, amsthm, amssymb}
\usepackage{enumerate}
\usepackage{pdflscape}
\usepackage{caption}
\usepackage{bm}

\usepackage{ifpdf}
\ifpdf
\usepackage[pdftex]{graphicx}
\else
\usepackage[dvips]{graphicx}
\fi
\usepackage{tikz}
 	 \usetikzlibrary{arrows,backgrounds}
\usepackage[all]{xy}

\usepackage{multicol}

\usepackage{tocvsec2}

\usepackage{bbm}

\input xy
\xyoption{all}

\usepackage[pdftex,plainpages=false,hypertexnames=false,pdfpagelabels]{hyperref}
\newcommand{\arxiv}[1]{\href{http://arxiv.org/abs/#1}{\tt arXiv:\nolinkurl{#1}}}
\newcommand{\arXiv}[1]{\href{http://arxiv.org/abs/#1}{\tt arXiv:\nolinkurl{#1}}}
\newcommand{\doi}[1]{\href{http://dx.doi.org/#1}{{\tt DOI:#1}}}

\newcommand{\googlebooks}[1]{(preview at \href{http://books.google.com/books?id=#1}{google books})}

\usepackage{xcolor}
\definecolor{dark-red}{rgb}{0.7,0.25,0.25}
\definecolor{dark-blue}{rgb}{0.15,0.15,0.55}
\definecolor{medium-blue}{rgb}{0,0,.8}
\definecolor{DarkGreen}{RGB}{0,150,0}
\definecolor{rho}{named}{red}
\hypersetup{
   colorlinks, linkcolor={purple},
   citecolor={medium-blue}, urlcolor={medium-blue}
}

\usepackage{longtable}
\usepackage{fullpage}

\theoremstyle{plain}
\newtheorem{thm}{Theorem}[section]
\newtheorem*{thm*}{Theorem}
\newtheorem{thmalpha}{Theorem}

\newtheorem{cor}[thm]{Corollary}

\newtheorem*{cor*}{Corollary}

\newtheorem*{conj*}{Conjecture}
\newtheorem{lem}[thm]{Lemma}

\newtheorem{prop}[thm]{Proposition}

\newtheorem*{quest*}{Question}
\newtheorem*{claim*}{Claim}

\theoremstyle{definition}
\newtheorem{defn}[thm]{Definition}

\newtheorem{alg}[thm]{Algorithm}

\newtheorem{nota}[thm]{Notation}

\newtheorem{ex}[thm]{Example}
\newtheorem{sub-ex}[thm]{Sub-Example}
\newtheorem{rem}[thm]{Remark}
\newtheorem*{rem*}{Remark}

\DeclareMathOperator{\coev}{coev}

\DeclareMathOperator{\End}{End}
\DeclareMathOperator{\ev}{ev}
\DeclareMathOperator{\Hom}{Hom}

\DeclareMathOperator{\op}{op}
\DeclareMathOperator{\ONB}{ONB}

\DeclareMathOperator{\id}{id}
\DeclareMathOperator{\Isom}{Isom}
\DeclareMathOperator{\ind}{ind}
\DeclareMathOperator{\im}{im}
\DeclareMathOperator{\Irr}{Irr}

\newcommand{\comment}[1]{}

\newcommand{\be}{\begin{enumerate}[label=(\arabic*)]}
\newcommand{\ee}{\end{enumerate}}

\newcommand{\set}[2]{\left\{#1 \middle| #2\right\}}

\newcommand{\WStar}{\bfW\text{*}}

\def\semicolon{;}
\def\applytolist#1{
    \expandafter\def\csname multi#1\endcsname##1{
        \def\multiack{##1}\ifx\multiack\semicolon
            \def\next{\relax}
        \else
            \csname #1\endcsname{##1}
            \def\next{\csname multi#1\endcsname}
        \fi
        \next}
    \csname multi#1\endcsname}

\def\calc#1{\expandafter\def\csname c#1\endcsname{{\mathcal #1}}}
\applytolist{calc}QWERTYUIOPLKJHGFDSAZXCVBNM;
\def\bbc#1{\expandafter\def\csname bb#1\endcsname{{\mathbb #1}}}
\applytolist{bbc}QWERTYUIOPLKJHGFDSAZXCVBNM;
\def\bfc#1{\expandafter\def\csname bf#1\endcsname{{\mathbf #1}}}
\applytolist{bfc}QWERTYUIOPLKJHGFDSAZXCVBNM;
\def\sfc#1{\expandafter\def\csname s#1\endcsname{{\sf #1}}}
\applytolist{sfc}QWERTYUIOPLKJHGFDSAZXCVBNM;
\def\fc#1{\expandafter\def\csname f#1\endcsname{{\mathfrak #1}}}
\applytolist{fc}QWERTYUIOPLKJHGFDSAZXCVBNM;

\newcommand{\spbfBim}{{\sf Bim_{bf}^{sp}}}
\renewcommand{\Vec}{{\sf Vec}}

\newcommand{\Hilb}{{\sf Hilb}}

\newcommand{\noshow}[1]{}
\newcommand{\MR}[1]{}

\usetikzlibrary{shapes}
\usetikzlibrary{cd}
\usetikzlibrary{backgrounds}
\usetikzlibrary{decorations,decorations.pathreplacing,decorations.markings}
\usetikzlibrary{fit,calc,through}
\usetikzlibrary{external}
\usetikzlibrary{arrows}
\tikzset{vertex/.style = {shape=circle,draw,fill=black,inner sep=0pt,minimum size=5pt}}
\tikzset{edge/.style = {->,> = latex', bend right}}
\tikzset{
	super thick/.style={line width=3pt}
}
\tikzset{
    quadruple/.style args={[#1] in [#2] in [#3] in [#4]}{
        #1,preaction={preaction={preaction={draw,#4},draw,#3}, draw,#2}
    }
}
\tikzstyle{shaded}=[fill=red!10!blue!20!gray!30!white]
\tikzstyle{unshaded}=[fill=white]
\tikzstyle{empty box}=[circle, draw, thick, fill=white, opaque, inner sep=2mm]
\tikzstyle{annular}=[scale=.7, inner sep=1mm, baseline]
\tikzstyle{rectangular}=[scale=.75, inner sep=1mm, baseline=-.1cm]
\tikzstyle{mid>}=[decoration={markings, mark=at position 0.5 with {\arrow{>}}}, postaction={decorate}]
\tikzstyle{mid<}=[decoration={markings, mark=at position 0.5 with {\arrow{<}}}, postaction={decorate}]
\tikzstyle{over}=[double, draw=white, super thick, double=]

\newcommand{\roundNbox}[6]{
	\draw[rounded corners=5pt, very thick, #1] ($#2+(-#3,-#3)+(-#4,0)$) rectangle ($#2+(#3,#3)+(#5,0)$);
	\coordinate (ZZa) at ($#2+(-#4,0)$);
	\coordinate (ZZb) at ($#2+(#5,0)$);
	\node at ($1/2*(ZZa)+1/2*(ZZb)$) {#6};
}

\newcommand{\ncircle}[4]{
	\draw[very thick, #1] #2 circle (#3);
	\node at #2 {#4};
}

\begin{document}
\title{Q-systems and compact W*-algebra objects}
\author{Corey Jones and David Penneys}
\date{\today}
\maketitle
\begin{abstract}
We show that given a rigid C*-tensor category, there is an equivalence of categories between normalized irreducible Q-systems, also known as connected unitary Frobenius algebra objects, and compact connected W*-algebra objects.
Although this result could be proved as a corollary of our previous article on realizations of algebra objects and discrete subfactors, we prove it here directly via categorical methods without passing through subfactor theory.
%
\end{abstract}


\section{Introduction}

The standard invariant of a finite index subfactor ${\rm II}_1$ subfactor was first defined as a $\lambda$-lattice \cite{MR1334479}.
Since, it has been reinterpreted as a planar algebra \cite{math.QA/9909027} and a Q-system, or unitary Frobenius algebra object, in a rigid C*-tensor category \cite{MR1027496,MR1966524,MR3308880}.
Using this third language, the standard invariant of a finite index irreducible ${\rm II}_1$ subfactor is an irreducible Q-system, i.e., the underlying algebra object $A$ is \emph{connected}, meaning $\dim(\cC(1_\cC, A)) = 1$.

In \cite{1611.04620}, we introduced the notion of a W*-algebra object in a rigid C*-tensor category.
In \cite{1704.02035}, we characterized an extremal irreducible discrete inclusion of factors $(N\subseteq M, E)$ where $N$ is type ${\rm II}_1$ and $E: M \to N$ is a faithful normal conditional expectation in terms of a connected W*-algebra object $\bfA$ in (an ind-completion of) a rigid C*-tensor category $\cC$, and a fully faithful embedding $\cC \hookrightarrow \spbfBim(N)$, the spherical/extremal bifinite bimodules over $N$.
This means that we may also view the standard invariant of a finite index irreducible ${\rm II}_1$ subfactor as a \emph{compact} connected W*-algebra object, which actually lies not only in the ind-completion of $\cC$, but also in the rigid involutive tensor subcategory $\cC^{\natural}$ which is obtained from $\cC$ simply by forgetting the adjoint.
Thus passing through our theorem \cite[Thm.~A and \S7.2]{1704.02035}, we would get an equivalence between irreducible Q-systems and compact connected W*-algebra objects.

The purpose of the present article is to provide a direct proof of the equivalence between irreducible Q-systems and compact connected W*-algebra objects in an arbitrary rigid C*-tensor category without passing through subfactor theory.
In doing so, we observe that the correct notion of morphisms between Q-systems is that of involutive algebra morphism from \cite[Def.~3.4 and 3.9]{MR2794547}.
We refer the reader to Section \ref{sec:Q-systemMorphisms} below for the definition.

Our main theorem is:

\begin{thmalpha}
\label{thm:Main}
Let $\cC$ be a rigid \emph{C*}-tensor category.\footnote{In this article, we assume our rigid C*-tensor category $\cC$ has simple unit object and is idempotent complete/semi-simple.}
The assignment $\bfA \mapsto \bfL^2\bfA$ induces an equivalence of categories
\[
\left\{\,
\parbox{5.6cm}{\rm Compact connected W*-algebra objects $\bfA\in \Vec(\cC)$ with unital\,\, $*$-algebra natural transformations}\,\right\}
\,\,\cong\,\,
\left\{\,\parbox{6.5cm}{\rm Normalized irreducible Q-systems with involutive algebra morphisms}\,\right\}.
\]
\end{thmalpha}


The proof of Theorem \ref{thm:Main} proceeds as follows.
In Section \ref{sec:QtoW}, starting with an irreducible Q-system $A\in \cC$, we define a compact connected W*-algebra object $\bfA\in \Vec(\cC)$ by $\bfA(a) = \cC^{\natural}(a, A)$ for all $a\in \cC$.
The $*$-structure is the conjugation on $\cC^{\natural}$ as $A$ is a real (symmetrically self-dual) object by \cite[\S6]{MR1444286}.
In Section \ref{sec:WtoQ}, starting with $\bfA$, we embed the compact Hilbert space object $\bfL^2\bfA\in \Hilb(\cC)$ into the (non-irreducible) Q-system $\bfL^2\bfA \otimes \overline{\bfL^2\bfA}$, which gives a canonical projector $p \in \End_\cC(H\otimes \overline{H})$ where $H\in \cC$ is such that $\bfL^2\bfA(a) = \cC(a, H)$ for all $a\in \cC$.
This is similar to the strategy of \cite[Lem.~3.21]{MR2794547}, where it is shown that an involutive subalgebra of a unitary Frobenius algebra is again Frobenius.
Conceptually, this is analogous to the fact that a unital $*$-subalgebra of a finite dimensional von Neumann algebra is again a von Neumann algebra.
We then prove the equivalence of categories in Sections \ref{sec:WtoQtoW} and \ref{sec:QtoWtoQ}.

While our proof is conceptual from an operator algebraic perspective, and focuses completely on the algebra objects, it is important to note that there is a second conceptual approach that passes through proper module categories.
Taking this approach, this result should be viewed as an extension of \cite[Appendix A]{1704.04729}.
One could expand the proof of \cite[Thm.~A.1]{1704.04729} to prove an equivalence between the category of irreducible Q-systems and the category of proper cyclic $\cC$-module W*-categories with simple basepoint.
One would then apply our equivalence between W*-algebra objects and cyclic $\cC$-module W*-categories \cite[Thm.~3.24]{1611.04620} to get another proof of Theorem \ref{thm:Main}.

\paragraph{Acknowledgements}

The authors would like to thank Marcel Bischoff, Andr\'{e} Henriques, and Jamie Vicary for helpful discussions.
Corey Jones was supported by Discovery Projects `Subfactors and symmetries' DP140100732 and `Low dimensional categories' DP160103479 from the Australian Research Council.
David Penneys was supported by NSF grants DMS-1500387/1655912 and 1654159.

The authors would like to thank the Isaac Newton Institute for Mathematical Sciences, Cambridge, for support and hospitality during the programme Operator Algebras: Subfactors and their Applications where work on this paper was undertaken.
This work was supported by EPSRC grant no EP/K032208/1.
\section{Background}

We refer the reader to \cite[\S2.1-2.2]{MR3663592} and \cite[\S2.1-2.3]{1611.04620} for background on rigid C*-tensor categories and their graphical calculus.
Of particular importance is the \emph{bi-involutive structure} consisting of two commuting involutions called the adjoint $*$ and the conjugate $\overline{\,\cdot\,}$.
Applying both to morphisms in $\cC$ gives the contravariant dual functor $(\cdot)^\vee: \cC \to \cC$:
\begin{equation}
\label{eq:BarStar}
\cC(a,b)\ni\psi \mapsto 
\psi^\vee := \overline{\psi}^* = \overline{\psi^*}=
\begin{tikzpicture}[baseline=-.1cm]
	\draw (0,.2) arc (0:180:.2cm) -- (-.4,-.5);
	\draw (0,-.2) arc (-180:0:.2cm) -- (.4,.5);
	\roundNbox{unshaded}{(0,0)}{.25}{0}{0}{$\psi$}
	\node at (-.6,-.4) {\scriptsize{$\overline{b}$}};
	\node at (.6,.4) {\scriptsize{$\overline{a}$}};
	\node at (-.15,-.4) {\scriptsize{$a$}};
	\node at (.15,.4) {\scriptsize{$b$}};
\end{tikzpicture}
\in \cC(\overline{b}, \overline{a}).
\end{equation}
The equation \eqref{eq:BarStar} implies that we actually have \emph{three} commuting involutions on morphisms: $(\cdot)^*, \overline{\,\cdot\,}, (\cdot)^\vee$.
Here, $(\cdot)^*, \overline{\,\cdot\,}$ are anti-linear and $(\cdot)^\vee$ is linear, and combining any two in either order gives the third.
As usual, we suppress all associators, unitors, and the conjugation structure maps.

\subsection{Tensor categories associated to \texorpdfstring{$\cC$}{C}}

Fix a rigid C*-tensor category $\cC$, and let $\Irr(\cC)$ be a set of representatives for the isomorphism classes of simple objects of $\cC$.

\begin{defn}[{\cite[\S2]{MR1444286}}]
\label{defn:StandardSolutions}
We say that morphisms $R_c: 1_\cC \to \overline{c}\otimes c$ and $S_c: 1_\cC \to c\otimes \overline{c}$ are \emph{solutions to the conjugate equations} if they satisfy the zig-zag axioms.
They are called \emph{standard solutions} if they additionally satisfy the balancing condition
$$
R_c^* \circ (\id_{\overline{c}} \otimes f) \circ R_c 
=
S_c^* \circ (f\otimes \id_{\overline{c}} ) \circ S_c 
\qquad\qquad
\text{for all $f\in \End_\cC(c)$.}
$$
\end{defn}

\begin{defn}
We rapidly recall the tensor categories associated to $\cC$ which arise in the study of $*$-algebra objects.
We refer the reader to \cite{1611.04620,1704.02035} for more details on these categories.
\begin{itemize}
\item
$\cC^{\natural}$ is the rigid involutive tensor category obtained from $\cC$ by forgetting the adjoint $*$.
\item
$\Vec(\cC)$ is the tensor category of linear functors $(\cC^{\natural})^{\op} \to \Vec$ with linear natural transformations and the Day convolution tensor product.
Note that $\Vec(\cC)$ is equivalent to $\ind(\cC^{\natural})$.
\item
$\Hilb(\cC)$ is the W*-tensor category of linear dagger tensor functors $\cC^{\op} \to \Hilb$ with bounded linear natural transformations and the Day convolution tensor product.
Note that $\Hilb(\cC)$ is equivalent to the Neshveyev-Yamashita unitary $\ind$ category of $\cC$ \cite{MR3509018}.
\end{itemize}
\end{defn}

Recall that the Yoneda embedding $a \mapsto \mathbf{a}:=\cC^{\natural}(\cdot , a)$ gives an embedding of involutive tensor categories $\cC^{\natural} \hookrightarrow \Vec(\cC)$ and similalry, we have an embedding of bi-involutive tensor categories $\cC\hookrightarrow \Hilb(\cC)$.

An object $\bfF\in \Vec(\cC)$ is called \emph{compact} if $\bfF$ is in the image of the Yoneda embedding, and similarly for $\bfH \in \Hilb(\cC)$.
This means that for $\bfF$ compact, there is a corresponding $F\in \cC^{\natural}$ such that $\bfF(a) = \cC^{\natural}(a, F)$ for all $a\in \cC$.
Likewise, we may identify compact objects of $\Hilb(\cC)$ with objects in $\cC$.

\begin{nota}
Typically, we use bold letters for functors, like $\bfF\in \Vec(\cC)$ and $\bfH \in \Hilb(\cC)$, and we will use regular letters for objects in $\cC^{\natural}$ and $\cC$.
\end{nota}

\subsection{C* Frobenius algebra objects}

We now introduce C* Frobenius algebra objects in rigid C*-tensor categories following \cite[\S3.1]{MR3308880}.
Of particular importance will be Q-systems, which are unitary Frobenius algebra objects discussed in Section \ref{sec:Q-systemMorphisms}. 
Some other general references on Frobenius algebras (and Q-systems) in rigid C*-tensor categories include 
\cite{MR1027496,MR1444286,MR1966524,MR2794547}.

\begin{defn}
An \emph{algebra object} in a tensor category $\cT$ is a triple $(A, m,i)$ where $A\in \cT$ and $m\in \cT( A\otimes A , A)$ and $i\in \cT(1_\cT , A)$ are morphisms satisfying the following relations:
\begin{itemize}
\item
(associativity)
$m\circ (m \otimes \id_A) = m\circ (\id_A \otimes m)$.
\item
(unitality)
$m\circ (i\otimes \id_A) = \id_A = m\circ (\id_A\otimes i)$.
\end{itemize}
A coalgebra object is defined similarly by reversing the arrows.

An algebra object is called \emph{connected} if $\dim(\cT(1_\cT, A)) = 1$.
\end{defn}

We make heavy use of the graphical calculus for morphisms in tensor categories, in which $m$ is denoted by a trivalent vertex, and $i$ is denoted by a univalent vertex. 
We now specialize to algebras in a rigid C*-tensor category $\cC$.

\begin{defn}
\label{defn:QSystem}
A C* \emph{Frobenius algebra object} is an algebra object $(A, m,i) \in \cC$ 
(which automatically implies $(A, m^*, i^*)$ is a coalgebra object)
such that the following condition holds:
\begin{itemize}
\item (Frobenius) 
$m^*\circ m = (\id_A \otimes m )\circ(m^*\otimes \id_A) = (m\otimes \id_A)\circ (\id_A \otimes m^*)$.
\end{itemize}
A C* Frobenius algebra is called \emph{separable} \cite[Def.~3.13]{MR1966524} or \emph{special} \cite[Def.~3.2]{MR3308880} if $m\circ m^* = \lambda \id_A$ for some constant $\lambda\in \bbC$. 
In this case, since $\cC$ is C*, we automatically have $\lambda>0$.
Since $1_\cC$ is simple and $\cC$ is C*, we automatically have $i^*\circ i = \lambda' \id_{1_\cC}$ for some non-zero constant $\lambda'>0$.
We call a C* Frobenius algebra \emph{normalized} if $\lambda'=1$.
\end{defn}

\begin{ex}
\label{ex:InnerHom}
Suppose $c\in \cC$, and consider $c\otimes \overline{c}$ with multiplication $m = \id_c \otimes \ev_c \otimes \id_{\overline{c}}$ and unit $i= \coev_c$.
Then $(c\otimes \overline{c}, m,i)$ is a separable C* Frobenius algebra, connected if and only if $c$ is simple.
Moreover, it is well known that $c\otimes \overline{c}$ is a real/symmetrically self-dual object.
\end{ex}

\begin{lem}
\label{lem:RotationRelation}
Suppose $(A,m,i)$ is an algebra object in $\cC$.
The Frobenius condition above is equivalent to the following two conditions:
\begin{itemize}
\item (Self-duality) 
There exist maps $\tilde{\ev}_A: A\otimes A \to 1_\cC$ and $\tilde{\coev}_A: 1_\cC \to A\otimes A$ which satisfy the zig-zag axioms, and
\item (Rotational invariance)
$m^* = (\id_A \otimes m)\circ (\tilde{\coev}_A\otimes \id_A) = (m\otimes \id_A)\circ (\id_A\otimes \tilde{\coev}_A)$.
\end{itemize}
\end{lem}
\begin{proof}
One passes from the Frobenius condition to the above two conditions by setting $\tilde{\ev}_A = i^*\circ m$ and $\tilde{\coev}_A=m^*\circ i$ and using the Frobenius relation.
One passes from the two above conditions to the Frobenius condition by trading $m^*$ for $m$ using the rotation relation and using associativity.
\end{proof}

\subsection{Unitary Frobenius algebras/Q-systems and involutive morphisms}
\label{sec:Q-systemMorphisms}

\begin{defn}
Given two algebra objects $(A,m_A,i_A), (B,m_B,i_B)$ in a tensor category $\cT$, an morphism $\theta\in \cT(A,B)$ is called an \emph{algebra morphism} if $m_B\circ (\theta \otimes \theta) = \theta \circ m_A$ and $i_B \circ \theta = i_A$.
\end{defn}

Given a C* Frobenius algebra $A\in\cC$, following \cite[Def.~3.4]{MR2794547},\footnote{To translate between our conventions for bi-involutive categories and those in \cite{MR2794547}, our $(\cdot)^*$ is his $(\cdot)^\dagger$, our $\overline{\,\cdot\,}$ is his $(\cdot)_*$, and our $(\cdot)^\vee$ is his $(\cdot)^*$.} we can define two canonical isomorphisms $\sigma_A^L, \sigma^R_A \in \cC( A, \overline{A})$ given by
$$
\sigma_A^L = 
\begin{tikzpicture}[baseline=-.1cm]
	\draw (-.3,.3) -- (-.3,.6);
	\filldraw (-.3,.6) circle (.05cm) node [left] {\scriptsize{$i_A^*$}};
	\filldraw (-.3,.3) circle (.05cm) node [below] {\scriptsize{$m_A$}};
	\draw (-.6,-.6) -- (-.6,0) arc (180:0:.3cm)  arc (-180:0:.3cm) -- (.6,.6);
	\node at (-.8,-.4) {\scriptsize{$A$}};
	\node at (.8,.4) {\scriptsize{$\overline{A}$}};
\end{tikzpicture}
\qquad\qquad
\sigma_A^R 
=
\begin{tikzpicture}[baseline=-.1cm, xscale = -1]
	\draw (-.3,.3) -- (-.3,.6);
	\filldraw (-.3,.6) circle (.05cm) node [left] {\scriptsize{$i_A^*$}};
	\filldraw (-.3,.3) circle (.05cm) node [below] {\scriptsize{$m_A$}};
	\draw (-.6,-.6) -- (-.6,0) arc (180:0:.3cm)  arc (-180:0:.3cm) -- (.6,.6);
	\node at (-.8,-.4) {\scriptsize{$A$}};
	\node at (.8,.4) {\scriptsize{$\overline{A}$}};
\end{tikzpicture}.
$$
Notice that the cup in the diagram on the left is $\coev_A$, and the cup in the diagram on the right is $\ev_A^*$.
By \cite[Lem.~3.5(13)]{MR2794547}, $\overline{\sigma_A^L}\circ \sigma_A^L = \varphi_A$ and $\overline{\sigma_A^R}\circ \sigma_A^R = \varphi_A$, and thus both are involutions.
By \cite[Lem.~3.11]{MR2794547}, $\sigma_A^L = \sigma_A^R$ if and only if either $\sigma_A^L$ or $\sigma_A^R$ is unitary.

\begin{lem}
\label{lem:UnitaryInvolution}
The following are equivalent for a separable \emph{C*} Frobenius algebra object $(A,m,i)\in\cC$.
\begin{enumerate}[(1)]
\item
$R=S=i_A^*\circ m_A$ are standard solutions to the conjugate equations.
\item
$R^*\circ R = i^*\circ m \circ m^*\circ i = d_A \id_{1_\cC}$
\item
$\sigma^L_A$ is unitary, and thus $\sigma_A^R=\sigma_A^L=:\sigma_A$ by \cite[Lem.~3.11]{MR2794547}.
\end{enumerate}
\end{lem}
\begin{proof}
(1) implies (2) is immediate, since (2) is satisfied by standard solutions.
(2) implies (1) follows from \cite[\S6]{MR1444286}.\footnote{We warn the reader that the proof in \cite[\S6]{MR1444286} for a separable C* Frobenius algebra involves constructing a subfactor, and thus the Jones index is the minimal index \cite{MR1027496}.}
To show (1) implies (3), suppose $R=S= i_A^*\circ m_A$ are standard solutions to the conjugate equations.
Then so are $R'=\ev_A^*$ and $S' = \coev_A$, and we must have $\sigma_A^L = (R\otimes \id_A)\circ( \id_A\otimes \coev_A)$ is unitary by \cite[Prop.~2.2.13]{MR3204665}.
Finally, to show (3) implies (2), if $\sigma^L_A$ is unitary, then it is easy to calculate that for $R=S= i_A^*\circ m_A$, \begin{equation*}
R^*\circ R
=
\begin{tikzpicture}[baseline=-.1cm]
	\draw (-.3,.6) -- (-.3,.9);
	\draw (-.3,-.6) -- (-.3,-.9);
	\filldraw (-.3,.9) circle (.05cm) node [left] {\scriptsize{$i_A^*$}};
	\filldraw (-.3,.6) circle (.05cm) node [below] {\scriptsize{$m_A$}};
	\filldraw (-.3,-.9) circle (.05cm) node [left] {\scriptsize{$i_A$}};
	\filldraw (-.3,-.6) circle (.05cm) node [above] {\scriptsize{$m_A^*$}};
	\draw (-.6,-.3) -- (-.6,.3) arc (180:0:.3cm) -- (0,-.3);
	\draw (-.6,-.3) arc (-180:0:.3cm);
\end{tikzpicture}
=
\begin{tikzpicture}[baseline=-.1cm]
	\draw (-.3,.6) -- (-.3,.9);
	\draw (-.3,-.6) -- (-.3,-.9);
	\filldraw (-.3,.9) circle (.05cm) node [left] {\scriptsize{$i_A^*$}};
	\filldraw (-.3,.6) circle (.05cm) node [below] {\scriptsize{$m_A$}};
	\filldraw (-.3,-.9) circle (.05cm) node [left] {\scriptsize{$i_A$}};
	\filldraw (-.3,-.6) circle (.05cm) node [above] {\scriptsize{$m_A^*$}};
	\draw (-.6,-.3) -- (-.6,.3) arc (180:0:.3cm)  arc (-180:0:.2cm) -- (.4,.9) arc (180:0:.2cm) -- (.8,-.9);
	\draw (-.6,-.3) arc (-180:0:.3cm) arc (180:0:.2cm) -- (.4,-.9) arc (-180:0:.2cm);
	\node at (.2,.9) {\scriptsize{$\overline{A}$}};
	\node at (.2,-.9) {\scriptsize{$\overline{A}$}};
	\node at (1,0) {\scriptsize{$A$}};
\end{tikzpicture}
=
\begin{tikzpicture}[baseline=-.1cm]
	\draw (0,.8) arc (180:0:.3cm) -- (.6,-.8) arc (0:-180:.3cm) -- (0,.8);
	\roundNbox{unshaded}{(0,.5)}{.3}{0}{0}{$\sigma_A$}
	\roundNbox{unshaded}{(0,-.5)}{.3}{0}{0}{$\sigma_A^*$}
	\node at (-.2,1) {\scriptsize{$\overline{A}$}};
	\node at (-.2,-1) {\scriptsize{$\overline{A}$}};
	\node at (.8,0) {\scriptsize{$A$}};
	\node at (-.2,0) {\scriptsize{$A$}};
\end{tikzpicture}
\underset{(3)}{=}
\begin{tikzpicture}[baseline=-.1cm]
	\draw (0,0) circle (.3cm);
	\node at (.5,0) {\scriptsize{$A$}};
\end{tikzpicture}
=
d_A \id_{1_\cC}.
\qedhere
\end{equation*}
\end{proof}

\begin{defn}
A \emph{unitary Frobenius algebra} or a \emph{Q-system} is a separable C* Frobenius algebra such that the equivalent conditions in Lemma \ref{lem:UnitaryInvolution} hold \cite[Def.~3.2 and 3.8]{MR3308880}.
A Q-system is called \emph{irreducible} if the underlying algebra object $A\in \cC$ is connected.
\end{defn}

\begin{rem}
\label{rem:RealObject}
By condition (1) in Lemma \ref{lem:UnitaryInvolution}, the underlying object $A\in \cC$ of a Q-system is \emph{real} or \emph{symmetrically self-dual} \cite[Rem.~5.6(3)]{MR1966524}.
Note that a Q-system is exactly a unitary $\dagger$-Frobenius monoid in the sense of \cite[Def.~3.9]{MR2794547}.
\end{rem}

\begin{cor}
A connected \emph{C*} Frobenius algebra is automatically an irreducible Q-system.
\end{cor}
\begin{proof}
Let $(A,m,i)$ be a connected C* Frobenius algebra.
By \cite[Lem.~3.3]{MR3308880}, $A$ is separable.
Moreover, $R=S= m^*\circ i$ are automatically standard solutions by \cite[\S6]{MR1444286}.\footnote{This was first observed by \cite[Rem.~5.6(3)]{MR1966524}.
Note that the proof in \cite[\S6]{MR1444286} relies on subfactor theory.
The more general result \cite[Thm.~2.9]{1704.04729} does not pass through subfactor theory and implies $R,S$ are standard solutions.}
\end{proof}

\begin{ex}
\label{ex:InnerHomQSystem}
The separable C* Frobenius algebra $(c\otimes \overline{c}, m,i)$ from Example \ref{ex:InnerHom} is a Q-system since $R^*\circ R = d_c^2$.
However, one must replace the multiplication and unit by $m = d_c^{1/2}(\id_c \otimes \ev_c \otimes \id_{\overline{c}})$ and $i= d_c^{-1/2}\coev_c$ to get a normalized Q-system.
As before, the Q-system is irreducible if and only if $c$ is simple.
\end{ex}

The next definition is adapted from \cite[Def.~2.21]{MR2794547} and gives the correct notion of morphism between irreducible Q-systems for our purposes.

\begin{defn}
Suppose $(A, m_A, i_A) , (B,m_B,i_B)$ are algebra objects in $\cC$.
An algebra morphism $\theta \in \cC( A, B)$ is called \emph{involutive} if $\sigma_B \circ \theta = \overline{\theta} \circ \sigma_A$.
\end{defn}

The recent article  \cite{1704.04729} introduces another notion of isomorphism between (unitary) C*-Frobenius algebras.
The say an algebra isomorphism $\theta \in \cC( A, B)$ is a C*-Frobenius algebra isomorphism if 
\begin{equation}
\label{eq:UnitaryFrobeniusAlgebraIsomorphism}
m_A\circ [\id_A\otimes (\theta^*\circ \theta)] = \theta^*\circ \theta \circ m_A.
\end{equation}
We now prove that involutive algebra morphisms satisfy \eqref{eq:UnitaryFrobeniusAlgebraIsomorphism}.

\begin{lem}
Every involutive algebra morphism $\theta\in \cC(A, B)$ satisfies \eqref{eq:UnitaryFrobeniusAlgebraIsomorphism}.
\end{lem}
\begin{proof}
First, using that $\theta$ is involutive and \eqref{eq:BarStar}, we have
\begin{equation}
\label{eq:SwitchTheta}
\begin{tikzpicture}[baseline=.4cm]
	\draw (-.3,1.4) -- (-.3,.3) arc (-180:0:.3cm) -- (.3,1.4);
	\draw (0,0) -- (0,-.3);
	\filldraw (0,0) circle (.05cm) node [above] {\scriptsize{$m_A^*$}};
	\filldraw (0,-.3) circle (.05cm) node [right] {\scriptsize{$i_A$}};
	\roundNbox{unshaded}{(.3,.75)}{.25}{0}{0}{$\theta$}
	\node at (-.5,1.2) {\scriptsize{$A$}};
	\node at (.5,1.2) {\scriptsize{$B$}};
\end{tikzpicture}
=
\begin{tikzpicture}[baseline=.4cm]
	\draw (-1.5,1.4) -- (-1.5,-.3) arc (-180:0:.3cm) -- (-.9,.3) arc (180:0:.3cm) (-.3,.3) arc (-180:0:.3cm) -- (.3,1.4);
	\draw (0,0) -- (0,-.3);
	\filldraw (0,0) circle (.05cm) node [above] {\scriptsize{$m_A^*$}};
	\filldraw (0,-.3) circle (.05cm) node [right] {\scriptsize{$i_A$}};
	\roundNbox{unshaded}{(.3,.75)}{.25}{0}{0}{$\theta$}
	\node at (-1.7,1.2) {\scriptsize{$A$}};
	\node at (.5,1.2) {\scriptsize{$B$}};
	\node at (-1.1,.3) {\scriptsize{$\overline{A}$}};
\end{tikzpicture}
=
\begin{tikzpicture}[baseline=.4cm]
	\draw (-1.5,1.4) -- (-1.5,-.3) arc (-180:0:.3cm) -- (-.9,.3) arc (180:0:.3cm) (-.3,.3) arc (-180:0:.3cm) -- (.3,1.4);
	\draw (0,0) -- (0,-.3);
	\filldraw (0,0) circle (.05cm) node [above] {\scriptsize{$m_B^*$}};
	\filldraw (0,-.3) circle (.05cm) node [right] {\scriptsize{$i_B$}};
	\roundNbox{unshaded}{(-.9,0)}{.25}{0}{0}{$\overline{\theta}$}
	\node at (-1.7,1.2) {\scriptsize{$A$}};
	\node at (.5,1.2) {\scriptsize{$B$}};
	\node at (-1.1,.5) {\scriptsize{$\overline{B}$}};
	\node at (-.7,-.5) {\scriptsize{$\overline{A}$}};
\end{tikzpicture}
\underset{\eqref{eq:BarStar}}{=}
\begin{tikzpicture}[baseline=.4cm]
	\draw (-.3,1.4) -- (-.3,.3) arc (-180:0:.3cm) -- (.3,1.4);
	\draw (0,0) -- (0,-.3);
	\filldraw (0,0) circle (.05cm) node [above] {\scriptsize{$m_B^*$}};
	\filldraw (0,-.3) circle (.05cm) node [right] {\scriptsize{$i_B$}};
	\roundNbox{unshaded}{(-.3,.75)}{.25}{0}{0}{$\theta^*$}
	\node at (-.5,1.2) {\scriptsize{$A$}};
	\node at (.5,1.2) {\scriptsize{$B$}};
\end{tikzpicture}
.
\end{equation}
Now using the adjoint of \eqref{eq:SwitchTheta}, together with Lemma \ref{lem:RotationRelation}, and that $\theta^*\in \cC(B, A)$ is a coalgebra morphism, we have
\begin{equation*}
\begin{tikzpicture}[baseline=-.3cm]
	\draw (0,.6) -- (0,1);
	\filldraw (0,.6) circle (.05cm) node [below] {\scriptsize{$m_A$}};
	\draw (-.3,-1.4) -- (-.3,.3) arc (180:0:.3cm) -- (.3,-1.4);
	\roundNbox{unshaded}{(.3,-.1)}{.25}{0}{0}{$\theta^*$}
	\roundNbox{unshaded}{(.3,-.75)}{.25}{0}{0}{$\theta$}
	\node at (.2,.8) {\scriptsize{$A$}};
	\node at (.5,-1.2) {\scriptsize{$A$}};
	\node at (-.5,-1.2) {\scriptsize{$A$}};
\end{tikzpicture}
=
\begin{tikzpicture}[baseline=-.1cm]
	\draw (.3,-1.4) -- (.3,.3);
	\draw (-.6,-1.4) -- (-.6,.6) arc (180:0:.3cm) (0,.6) arc (-180:0:.3cm) -- (.6,1.4);
	\draw (-.3,.9) -- (-.3,1.2);
	\filldraw (.3,.3) circle (.05cm) node [above] {\scriptsize{$m_A^*$}};
	\filldraw (-.3,.9) circle (.05cm) node [below] {\scriptsize{$m_A$}};
	\filldraw (-.3,1.2) circle (.05cm) node [left] {\scriptsize{$i_A^*$}};
	\roundNbox{unshaded}{(.3,-.1)}{.25}{0}{0}{$\theta^*$}
	\roundNbox{unshaded}{(.3,-.75)}{.25}{0}{0}{$\theta$}
	\node at (.8,1.2) {\scriptsize{$A$}};
	\node at (.5,-1.2) {\scriptsize{$A$}};
	\node at (-.8,-1.2) {\scriptsize{$A$}};
\end{tikzpicture}
=
\begin{tikzpicture}[baseline=-.2cm]
	\draw (.3,-1.6) -- (.3,-.5);
	\draw (-.6,-1.6) -- (-.6,.6) arc (180:0:.3cm) -- (0,-.2) arc (-180:0:.3cm) -- (.6,1.4);
	\draw (-.3,.9) -- (-.3,1.2);
	\filldraw (.3,-.5) circle (.05cm) node [above] {\scriptsize{$m_B^*$}};
	\filldraw (-.3,.9) circle (.05cm) node [below] {\scriptsize{$m_A$}};
	\filldraw (-.3,1.2) circle (.05cm) node [left] {\scriptsize{$i_A^*$}};
	\roundNbox{unshaded}{(0,.25)}{.25}{0}{0}{$\theta^*$}
	\roundNbox{unshaded}{(.6,.25)}{.25}{0}{0}{$\theta^*$}
	\roundNbox{unshaded}{(.3,-.95)}{.25}{0}{0}{$\theta$}
	\node at (.8,1.2) {\scriptsize{$A$}};
	\node at (.5,-1.4) {\scriptsize{$A$}};
	\node at (-.8,-1.4) {\scriptsize{$A$}};
\end{tikzpicture}
\underset{\eqref{eq:SwitchTheta}^*}{=}
\begin{tikzpicture}[baseline=-.3cm]
	\draw (.3,-1.4) -- (.3,-.3);
	\draw (-.6,-1.4) -- (-.6,0) arc (180:0:.3cm) (0,0) arc (-180:0:.3cm) -- (.6,1.1);
	\draw (-.3,.3) -- (-.3,.6);
	\filldraw (.3,-.3) circle (.05cm) node [above] {\scriptsize{$m_B^*$}};
	\filldraw (-.3,.3) circle (.05cm) node [below] {\scriptsize{$m_B$}};
	\filldraw (-.3,.6) circle (.05cm) node [left] {\scriptsize{$i_B^*$}};
	\roundNbox{unshaded}{(-.6,-.75)}{.25}{0}{0}{$\theta$}
	\roundNbox{unshaded}{(.6,.45)}{.25}{0}{0}{$\theta^*$}
	\roundNbox{unshaded}{(.3,-.75)}{.25}{0}{0}{$\theta$}
	\node at (.8,.9) {\scriptsize{$A$}};
	\node at (.5,-1.2) {\scriptsize{$A$}};
	\node at (-.8,-1.2) {\scriptsize{$A$}};
\end{tikzpicture}
=
\begin{tikzpicture}[baseline=-.3cm]
	\draw (0,0) -- (0,1.1);
	\filldraw (0,0) circle (.05cm) node [below] {\scriptsize{$m_B$}};
	\draw (-.3,-1.4) -- (-.3,-.3) arc (180:0:.3cm) -- (.3,-1.4);
	\roundNbox{unshaded}{(0,.45)}{.25}{0}{0}{$\theta^*$}
	\roundNbox{unshaded}{(.3,-.75)}{.25}{0}{0}{$\theta$}
	\roundNbox{unshaded}{(-.3,-.75)}{.25}{0}{0}{$\theta$}
	\node at (.2,.9) {\scriptsize{$A$}};
	\node at (.5,-1.2) {\scriptsize{$A$}};
	\node at (-.5,-1.2) {\scriptsize{$A$}};
\end{tikzpicture}
=
\begin{tikzpicture}[baseline=-.3cm]
	\draw (0,-.9) -- (0,.9);
	\filldraw (0,-.9) circle (.05cm) node [below] {\scriptsize{$m_A$}};
	\draw (-.3,-1.4) -- (-.3,-1.2) arc (180:0:.3cm) -- (.3,-1.4);
	\roundNbox{unshaded}{(0,.25)}{.25}{0}{0}{$\theta^*$}
	\roundNbox{unshaded}{(0,-.45)}{.25}{0}{0}{$\theta$}
	\node at (.2,.7) {\scriptsize{$A$}};
	\node at (.5,-1.2) {\scriptsize{$A$}};
	\node at (-.5,-1.2) {\scriptsize{$A$}};
\end{tikzpicture}
.
\qedhere
\end{equation*}
\end{proof}

\subsection{Compact connected W*-algebra objects}

An algebra object $\bfA \in \Vec(\cC)$ may be viewed as a lax monoidal functor $(\cC^{\natural})^{\op} \to \Vec$ by \cite[Prop.~3.3]{1611.04620}.
If $\bfA \in \Vec(\cC)$ is a compact algebra object, the corresponding object $A\in \cC^{\natural}$ satisfying $\bfA(a) = \cC^{\natural}(a,A)$ for $a\in \cC$ is easily seen to be an algebra object.

\begin{defn}
A $*$-structure on an algebra object $\bfA\in \Vec(\cC)$ is a collection of conjugate linear natural isomorphisms $j_a : \bfA(a) \to \bfA(\overline{a})$ satisfying the following axioms:
\begin{itemize}
\item 
(conjugate naturality) 
for all $\psi \in \cC(b, a)$ and $f\in \bfA(a)$, $j_b(f \circ \psi) = j_{a}(f) \circ \overline{\psi}$.
\item
(involutive)
for all $f\in \bfA(a)$, $j_{\overline{a}}(j_a(f)) = f$.
\item
(unital) $j_{1_\cC} (i) =i$.
\item
(monoidal)
for all $f\in \bfA(a)$ and $g\in \bfA(b)$, $j_{a\otimes b}(m \circ (f\otimes g) ) = m\circ (j_a(g) \otimes j_b(f))$.
\end{itemize}
\end{defn}

\begin{ex}
\label{ex:StarAlgebraFromInnerHom}
Suppose $c\in \cC$, and consider the normalized Q-system $(c\otimes \overline{c},m,i) \in \cC$ from Example \ref{ex:InnerHomQSystem}.
By the Yoneda embedding, $\mathbf{c}\otimes\overline{\mathbf{c}} : (\cC^{\natural})^{\op} \to \Vec$ is given by $(\mathbf{c}\otimes\overline{\mathbf{c}})(a) = \cC^{\natural}(a, c\otimes \overline{c})$.
The $*$-algebra structure on $\mathbf{c}\otimes\overline{\mathbf{c}}$ is given by 
\begin{align*}
\mu_{a,b} &: 
(\mathbf{c}\otimes\overline{\mathbf{c}})(a)\otimes (\mathbf{c}\otimes\overline{\mathbf{c}})(b) 
\to 
(\mathbf{c}\otimes\overline{\mathbf{c}})(a\otimes b)
&&
f\otimes g\mapsto d_c^{1/2}(\id_c \otimes \ev_c \otimes \id_{\overline{c}})\circ (f\otimes g)
\\
i &\in (\mathbf{c}\otimes\overline{\mathbf{c}})(1_\cC)=\cC^{\natural}(1_\cC, c\otimes \overline{c})
&&i=d_c^{-1/2}\coev_c
\\
j_a &: (\mathbf{c}\otimes\overline{\mathbf{c}})(a) \to (\mathbf{c}\otimes\overline{\mathbf{c}})(\overline{a})
&&
f\mapsto \overline{f}
\end{align*}
(this definition of $j$ suppresses the isomorphism between the real object $c\otimes \overline{c}$ and its conjugate).
Then $\mathbf{c}\otimes\overline{\mathbf{c}}$ is a compact $*$-algebra object in $\Vec(\cC)$ which is connected if and only if $c$ is simple.
The corresponding object in $\cC^{\natural}$ is obviously $c\otimes \overline{c}$.
\end{ex}

It was shown in \cite[Thm.~3.20]{1611.04620} that $*$-structures on algebra objects in $\Vec(\cC)$ are equivalent to dagger structures on the cyclic $\cC$-module category $(\cM_\bfA,m)$ whose objects are the $\mathbf{c}\otimes \bfA$ for $c\in \cC$ and whose morphisms are right $\bfA$-module morphisms.
Indeed, when $\bfA\in \Vec(\cC)$ is compact, we have
$$
\cM_\bfA(a\otimes \bfA, b\otimes \bfA)
\cong
\cC(a, b\otimes A)
$$
with composition given for $f\in \cC(a, b\otimes A)$ and $g\in \cC(b, c\otimes A)$ by
$$
g \circ_{\cM_\bfA} f := (\id_c \otimes m) \circ_\cC (g\otimes \id_{A}) \circ f.
$$
The dagger structure on $\cM_\bfA$ induced by $j$ is given on $f\in \cC(a, b\otimes A)$ by
$$
f^* = j_{\overline{b}\otimes a}\bigg( (\ev_b \otimes \id_A) \circ f\bigg) \circ (\coev_a \otimes \id_{\overline{b}}).
$$
Conversely, given a cyclic $\cC$-module dagger category $(\cM, m)$, we recover the lax monoidal functor $\bfA : (\cC^{\natural})^{\op} \to \Vec$ by $\bfA(a) = \cM(a\otimes m ,m)$, and the $*$-structure is given by 
$$
j_a(f) = (\ev_a \otimes \id_m)\circ (\id_{\overline{a}}\otimes f^*)
$$ 
for $f\in \bfA(a)$.
We see that $\bfA$ is compact if and only if $\cM$ is \emph{proper} \cite{1704.04729} or \emph{cofinite} \cite{1511.07982}, i.e., $\cM(a\otimes m, m)$ is always finite dimensional, and non-zero for only finitely many $a\in \Irr(\cC)$.

Just as being a C*-algebra is a property of a complex $*$-algebra and not extra structure, being a C*-category is a property of a dagger category and not extra structure.
We refer the reader to \cite[\S2.1]{1611.04620} for more details.

\begin{defn}
A $*$-algebra object $(\bfA, m, i, j)\in \Vec(\cC)$ is a C*/W*-algebra object if the cyclic $\cC$-module dagger category $\cM_\bfA$ is a C*/W*-category.
If $\bfA$ is locally finite, meaning $\bfA(a)$ is finite dimensional for all $a\in \cC$, then all morphisms spaces of $\cM_\bfA$ are finite dimensional.
In this case, $\cM_\bfA$ is a C*-category if and only if $\cM_\bfA$ is a W*-category.
\end{defn}

\begin{ex}
\label{ex:W*AlgebraFromInnerHom}
The $*$-algebra object $\mathbf{c}\otimes \overline{\mathbf{c}}=\cC^\natural(\cdot, c\otimes \overline{c})\in \Vec(\cC)$ from Example \ref{ex:StarAlgebraFromInnerHom} is a W*-algebra object.
To see this, we construct a $*$-algebra natural isomorphism $\pi:\bfB(\mathbf{c}) \Rightarrow\mathbf{c}\otimes \overline{\mathbf{c}}$, where the instance of $\mathbf{c}$ inside $\bfB(\mathbf{c})$ is the linear dagger functor $\mathbf{c} = \cC(\cdot, c) : \cC^{\op} \to \Hilb$, which is a Hilbert space object under the Yoneda embedding.
For $a\in \cC$ and $f\in \bfB(\mathbf{c})(a)=\cC(a\otimes c, c)$, we define
$$
\pi_a(f) = 
d_c^{-1/2}
\begin{tikzpicture}[baseline=-.1cm]
	\draw (0,0) -- (0,.7);
	\draw (-.3,-.3) -- (-.3,-.7);
	\draw (.3,-.3) arc (-180:0:.3cm) -- (.9,.7);
	\roundNbox{unshaded}{(0,0)}{.3}{.3}{.3}{$f$}
	\node at (-.2,.5) {\scriptsize{$c$}};
	\node at (-.5,-.5) {\scriptsize{$a$}};
	\node at (.1,-.5) {\scriptsize{$c$}};
	\node at (1.1,.5) {\scriptsize{$\overline{c}$}};
\end{tikzpicture}
\in 
\cC(a, c\otimes \overline{c}) = (\mathbf{c}\otimes \overline{\mathbf{c}})(a).
$$
The reader can check $\pi$ is a $*$-algebra natural isomorphism.
(It is similar to, but easier than, the proof of Theorem \ref{thm:W*toQandBack} below.)
Finally, $\bfB(\mathbf{c})$ is the W*-algebra object corresponding to the cyclic $\cC$-module W*-category $(\cC, c)$, and thus $\mathbf{c}\otimes \overline{\mathbf{c}}$ is W*.
\end{ex}

\subsection{Inner products on hom spaces of \texorpdfstring{$\cC$}{C}}

The morphism spaces of $\cC$ come equipped with inner products given by
\begin{equation}
\label{eq:CInnerProduct}
\langle f| g\rangle_{\cC(a,b)} = 
\begin{tikzpicture}[baseline=-.1cm]
	\draw (0,.8) arc (180:0:.3cm) -- (.6,-.8) arc (0:-180:.3cm) -- (0,.8);
	\roundNbox{unshaded}{(0,.5)}{.3}{0}{0}{$f^*$}
	\roundNbox{unshaded}{(0,-.5)}{.3}{0}{0}{$g$}
	\node at (-.2,-1) {\scriptsize{$a$}};
	\node at (-.2,0) {\scriptsize{$b$}};
	\node at (-.2,1) {\scriptsize{$a$}};
	\node at (.8,0) {\scriptsize{$\overline{a}$}};
\end{tikzpicture}.
\end{equation}
For all $a,b\in \cC$, we pick orthonormal bases $\ONB(a,b)$ for $\cC(a,b)$ under the inner product \eqref{eq:CInnerProduct}.

For each $c\in \Irr(\cC)$ and $a\in \cC$, we let $\Isom(c, a)= \set{\sqrt{d_c}f}{f\in \ONB(c,a)}$, which is a maximal set of isometries in $\cC(c, a)$ with mutually orthogonal ranges.
Note that $\Isom(c,a)$ is an orthonormal basis for $\cC(c,a)$ with the modified inner product given by
\begin{equation}
\label{eq:IsomInnerProduct}
\langle f|g\rangle_{\Isom} \id_c = f^*\circ g .
\end{equation}

\begin{defn}
The 1-\emph{click rotation} or \emph{Fourier transform} $\cF$ on $\cC(a, b\otimes c)$ for $a,b,c\in\cC$ is given by
\begin{equation*}
\cC(c, a\otimes b) \ni f 
\mapsto 
\begin{tikzpicture}[baseline=-.1cm]
	\draw (0,-.3) arc (0:-180:.3cm) -- (-.6,.7);
	\draw (-.15,.3) -- (-.15,.7);
	\draw (.15,.3) arc (180:0:.3cm) -- (.75,-.7);
	\roundNbox{unshaded}{(0,0)}{.3}{0}{0}{$f$}
	\node at (-.15,.9) {\scriptsize{$a$}};
	\node at (-.6,.9) {\scriptsize{$\overline{c}$}};
	\node at (.75,-.9) {\scriptsize{$\overline{b}$}};
\end{tikzpicture}
\in \cC(\overline{b}, \overline{c}\otimes a).
\end{equation*}
Note that if $a,b,c\in\Irr(\cC)$ and $f\in \Isom(c, a\otimes b)$, then a weighting is required to produce another isometry:
\begin{equation*}
\Isom(c, a\otimes b) \ni f 
\mapsto 
\left(
\frac{d_b}{d_c}
\right)^{1/2}
\cF(f)
\end{equation*}
produces an isometry in $\cC(\overline{b}, \overline{c}\otimes a)$, and thus 
$$\set{\left(
\frac{d_b}{d_c}
\right)^{1/2}\cF(f)}{f\in \Isom(c, a\otimes b)
}
\subset \cC(\overline{b}, \overline{c}\otimes a)
$$ 
is a new orthonormal basis under inner product \eqref{eq:IsomInnerProduct}.
\end{defn}

Suppose now that $\bfA\in \Vec(\cC)$ is a compact connected W*-algebra object corresponding to $A\in \cC^{\natural}$ such that $\bfA(a) = \cC^{\natural}(a, A)$ for $a\in \cC$.
Then $\bfA$ is tracial by \cite[Prop.~2.6]{1704.02035}, and we get an additional inner product on $\bfA(a)=\cC(a,A)$ by
\begin{equation}
\label{eq:L2AInnerProduct}
\langle f|g\rangle_a :=
\begin{tikzpicture}[baseline = -.1cm, xscale=-1]
    \draw (-.5,-.3) arc (-180:0:.5cm);
    \draw (-.5,.3) arc (180:0:.5cm);
    \filldraw (0,.8) circle (.05cm);
    \draw (0,.8) -- (0,1.2);
    \roundNbox{unshaded}{(-.5,0)}{.3}{0}{0}{$g$}
    \roundNbox{unshaded}{(.5,0)}{.3}{.2}{.2}{$j_a(f)$}
    \node at (-.65,-.5) {\scriptsize{$\mathbf{a}$}};
    \node at (-.65,.5) {\scriptsize{$\bfA$}};
    \node at (.65,-.5) {\scriptsize{$\overline{\mathbf{a}}$}};
    \node at (.65,.5) {\scriptsize{$\bfA$}};
    \node at (.15,1) {\scriptsize{$\bfA$}};
\end{tikzpicture}
\underset{\text{($\bfA$ tracial)}}{=}
\begin{tikzpicture}[baseline = -.1cm]
    \draw (-.5,-.3) arc (-180:0:.5cm);
    \draw (-.5,.3) arc (180:0:.5cm);
    \filldraw (0,.8) circle (.05cm);
    \draw (0,.8) -- (0,1.2);
    \roundNbox{unshaded}{(-.5,0)}{.3}{0}{0}{$g$}
    \roundNbox{unshaded}{(.5,0)}{.3}{.2}{.2}{$j_a(f)$}
    \node at (-.65,-.5) {\scriptsize{$\mathbf{a}$}};
    \node at (-.65,.5) {\scriptsize{$\bfA$}};
    \node at (.65,-.5) {\scriptsize{$\overline{\mathbf{a}}$}};
    \node at (.65,.5) {\scriptsize{$\bfA$}};
    \node at (.15,1) {\scriptsize{$\bfA$}};
\end{tikzpicture}
={}_a\langle g, f\rangle
\end{equation}
We form the compact Hilbert space object $\bfL^2\bfA \in \Hilb(\cC)$ by $\bfL^2\bfA(a) = \bfA(a)$ with inner product \eqref{eq:L2AInnerProduct}.

Now $\bfA$ has a canonical state corresponding to the unique state on $\bfA(1_\cC) = \bbC$ via $i_\bfA \mapsto 1_\bbC$.
The GNS representation of $\bfA$ gives a canonical faithful $*$-algebra natural transformation $\lambda : \bfA \Rightarrow \bfB(\bfL^2\bfA)$ \cite[Eq.~(20)]{1611.04620}.
Since $\bfA$ is connected, the inner product \eqref{eq:L2AInnerProduct} satisfies
\begin{equation}
\label{eq:InnerProductInHilbC}
\langle f|g\rangle_a \id_{\bfL^2\bfA}
=
\begin{tikzpicture}[baseline=-.1cm]
	\draw (-.15,-.5) -- (-.15,.5);
	\draw (.15,-1.2) -- (.15,1.2);
	\roundNbox{unshaded}{(0,.5)}{.3}{.35}{.35}{$\lambda_a(g)$}
	\roundNbox{unshaded}{(0,-.5)}{.3}{.35}{.35}{$\lambda_a(f)^*$}
	\node at (-.3,0) {\scriptsize{$\mathbf{a}$}};
	\node at (.6,0) {\scriptsize{$\bfL^2\bfA$}};
	\node at (.6,1) {\scriptsize{$\bfL^2\bfA$}};
	\node at (.6,-1) {\scriptsize{$\bfL^2\bfA$}};
\end{tikzpicture}
=
\begin{tikzpicture}[baseline=-.1cm]
	\draw (-.3,.8) arc (0:180:.3cm) -- (-.9,-.8) arc (-180:0:.3cm);
	\draw (.3,-1.2) -- (.3,1.2);
	\roundNbox{unshaded}{(0,.5)}{.3}{.3}{.3}{$\lambda_a(g)$}
	\roundNbox{unshaded}{(0,-.5)}{.3}{.3}{.3}{$\lambda_a(f)^*$}
	\node at (-1.1,0) {\scriptsize{$\overline{\mathbf{a}}$}};
	\node at (.75,0) {\scriptsize{$\bfL^2\bfA$}};
	\node at (.75,1) {\scriptsize{$\bfL^2\bfA$}};
	\node at (.75,-1) {\scriptsize{$\bfL^2\bfA$}};
\end{tikzpicture}
\,.
\end{equation}

Finally, given a $*$-algebra natural transformation $\theta: \bfA \Rightarrow \bfB$ of connected W*-algebra objects, notice that $\bfL^2\bfA(a)\ni f\mapsto \theta(f)\in \bfL^2\bfB(a)$ gives a canonical isometry $\bfL^2\theta: \bfL^2\bfA \Rightarrow \bfL^2\bfB$ since 
\begin{align}
\langle \theta_a(f)|\theta_a(g)\rangle_a^\bfB \, i_\bfB
&=
\begin{tikzpicture}[baseline = -.1cm, xscale=-1]
    \draw (-.5,-.3) arc (-180:0:.5cm);
    \draw (-.5,.3) arc (180:0:.5cm);
    \filldraw (0,.8) circle (.05cm);
    \draw (0,.8) -- (0,1.2);
    \roundNbox{unshaded}{(-.5,0)}{.3}{.2}{.1}{$\theta_a(g)$}
    \roundNbox{unshaded}{(1,0)}{.3}{.6}{.6}{$j^\bfB_a(\theta_{a}(f))$}
    \node at (-.65,-.5) {\scriptsize{$\mathbf{a}$}};
    \node at (-.65,.5) {\scriptsize{$\bfB$}};
    \node at (.65,-.5) {\scriptsize{$\overline{\mathbf{a}}$}};
    \node at (.65,.5) {\scriptsize{$\bfB$}};
    \node at (.15,1) {\scriptsize{$\bfB$}};
\end{tikzpicture}
=
\begin{tikzpicture}[baseline = -.1cm, xscale=-1]
    \draw (-.5,-.3) arc (-180:0:.5cm);
    \draw (-.5,.3) arc (180:0:.5cm);
    \filldraw (0,.8) circle (.05cm);
    \draw (0,.8) -- (0,1.2);
    \roundNbox{unshaded}{(-.5,0)}{.3}{.2}{.1}{$\theta_a(g)$}
    \roundNbox{unshaded}{(1,0)}{.3}{.6}{.6}{$\theta_{\overline{a}}(j^\bfA_a(f))$}
    \node at (-.65,-.5) {\scriptsize{$\mathbf{a}$}};
    \node at (-.65,.5) {\scriptsize{$\bfB$}};
    \node at (.65,-.5) {\scriptsize{$\overline{\mathbf{a}}$}};
    \node at (.65,.5) {\scriptsize{$\bfB$}};
    \node at (.15,1) {\scriptsize{$\bfB$}};
\end{tikzpicture}
=
\theta_{1_\cC}
\left(
\begin{tikzpicture}[baseline = -.1cm, xscale=-1]
    \draw (-.5,-.3) arc (-180:0:.5cm);
    \draw (-.5,.3) arc (180:0:.5cm);
    \filldraw (0,.8) circle (.05cm);
    \draw (0,.8) -- (0,1.2);
    \roundNbox{unshaded}{(-.5,0)}{.3}{0}{0}{$g$}
    \roundNbox{unshaded}{(.5,0)}{.3}{.3}{.3}{$j^\bfA_a(f)$}
    \node at (-.65,-.5) {\scriptsize{$\mathbf{a}$}};
    \node at (-.65,.5) {\scriptsize{$\bfA$}};
    \node at (.65,-.5) {\scriptsize{$\overline{\mathbf{a}}$}};
    \node at (.65,.5) {\scriptsize{$\bfA$}};
    \node at (.15,1) {\scriptsize{$\bfA$}};
\end{tikzpicture}
\right)
\notag
\\&=
\theta_{1_\cC}(\langle f| g\rangle_a^\bfA i_\bfA)
=
\langle f| g\rangle_a^\bfA \theta_{1_\cC}(i_\bfA)
=
\langle f| g\rangle_a^\bfA i_\bfB
\label{eq:Isometry}
\end{align}
Noticing that this argument was independent of the W*-algebra objects being compact, we conclude the following.

\begin{cor}
Any $*$-algebra natural transformation $\theta: \bfA \Rightarrow \bfB$ between connected \emph{W*}-algebra objects in $\Vec(\cC)$ is injective and induces an isometry $\bfL^2\theta: \bfL^2\bfA \Rightarrow \bfL^2\bfB$ with respect to the right (respectively left) inner products. 
\end{cor}

\section{From Q-systems to W*-algebra objects}
\label{sec:QtoW}

In this section, we define a functor $\WStar$ from normalized irreducible Q-systems to compact connected W*-algebra objects.

\subsection{Objects}

Suppose we have a normalized irreducible Q-system $(A, m,i)$ in $\cC$.
We now define a compact connected W*-algebra object $\WStar(A):=(\bfA , m, i, j) \in \cC^{\natural}$.
For $a\in \cC^{\natural}$, we define $\bfA(a) = \cC^{\natural}(a, A)$.
We define the laxitor $\mu_{a,b} : \bfA(a) \otimes \bfA(b) \to \bfA(a\otimes b)$ to be the map $f\otimes g \mapsto m\circ (f\otimes g)$ and the same unit map $i \in \bfA(1_\cC) = \cC(1_\cC, A)$.
It is straightforward to verify $(\bfA, \mu,i)$ is a lax monoidal functor $(\cC^{\natural})^{\op}\to \Vec$.
Since $A\in \cC$, we have $\bfA$ is compact, and $\bfA(1_\cC)$ is clearly one dimensional, and thus $\bfA$ is connected.

We now define the $*$-structure on $\bfA$.

\begin{defn}
We define a $*$-structure $j$ on $\bfA$ by 
$$
\cC^{\natural}(a, A)
=
\bfA(a)
\ni
f
\overset{j_a}{\longmapsto}
\sigma_A^{-1}
\circ
\overline{f}
\in 
\cC^{\natural}(\overline{a}, A),
$$
where $\sigma_A\in \cC(A, \overline{A})$ is the unitary isomorphism from Lemma \ref{lem:UnitaryInvolution}.
Using that $\overline{\sigma_A}\circ \sigma_A = \varphi_A$ by \cite[Lem.~3.5(13)]{MR2794547}, it is easy to see $(j_a)_{a\in \cC}$ satisfies the axioms of a $*$-structure on $\bfA$.
\end{defn}

\begin{prop}
\label{prop:FrobeniusToW*}
The compact connected $*$-algebra object $(\bfA, \mu, i, j) \in \Vec(\cC)$ is \emph{W*}. 
\end{prop}
\begin{proof}
Consider the cyclic $\cC$-module dagger category $(\cM_A, A)$ of free right $A$-modules in $\cC$ with basepoint $A\in \cC$ and dagger structure given by the adjoint in $\cC$.
That is, objects of $\cM_A$ are exactly the $c\otimes A$ for $c\in \cC$, and the morphisms are right $A$-module morphisms.
By Lemma \ref{lem:RotationRelation}, for every $g \in \cM_A(a\otimes A, b\otimes A)$, $g^* \in \cM_A(b\otimes A, a\otimes A)$.

Thus $\cM_A$ is a dagger sub-category of $\cC$ which is closed under finite direct sums.
This means that $\cM_A$ is automatically W* by the finite dimensional bicommutant theorem \cite[Thm.~3.2.1]{JonesVNA} which says a unital $*$-subalgebra of a finite dimensional von Neumann algebra is again a von Neumann algebra.
Indeed, for all $c\in \cC$,
$$
\cM_A(c\otimes A, c\otimes A) \subseteq \cC(c\otimes A, c\otimes A)
$$
is a unital $*$-subalgebra of a finite dimensional von Neumann algebra.
This is sufficient to show $\cM_A$ is W* by Roberts' $2\times 2$ trick \cite{MR808930} (see also \cite[\S2.1 and Prop.~3.26]{1611.04620}).

We claim that $(\bfA, \mu, i, j)$ is exactly the $*$-algebra corresponding to the cyclic $\cC$-module W*-category $(\cM_A,A)$ under \cite[Thm.~3.20]{1611.04620}.
First, the free module functor \cite{MR1936496,MR2863377,MR3578212} gives for all $a\in \cC$ a natural isomorphism
\begin{equation}
\label{eq:FreeModuleIsomorphism}
\bfA(a) = \cC(a, A) \cong \cM_A(a\otimes A, A),
\end{equation}
and this isomorphism is compatible with the multiplication and unit on $\bfA$ and composition and $\id_A$ in $(\cM_A, A)$.
We now show the $*$-structure of $\bfA$ and the dagger structure of $\cM_A$ are compatible.
By \cite[Fig.~4]{MR1936496}, the isomorphism \eqref{eq:FreeModuleIsomorphism} is given for all $a\in \cC$ and $f\in \bfA(a)$ by
\begin{equation}
\label{eq:ExplicitGNS}
\cC(a, A)
\ni f
\mapsto
\begin{tikzpicture}[baseline=-.1cm]
	\draw (-.3,-1) -- (-.3,-.6);
	\draw (.3,-1) -- (.3,1);
	\filldraw (.3,.6) circle (.05cm);
	\draw (-.3,0) .. controls ++(90:.5cm) and ++(-135:.5cm) .. (.3,.6);
	\roundNbox{unshaded}{(-.3,-.3)}{.3}{0}{0}{$f$}
	\node at (.5,-.8) {\scriptsize{$A$}};
	\node at (-.5,-.8) {\scriptsize{$a$}};
	\node at (-.5,.2) {\scriptsize{$A$}};
	\node at (.5,.8) {\scriptsize{$A$}};
\end{tikzpicture}
\in \cM_A(a\otimes A, A).
\end{equation}
By \cite[Thm.~3.20]{1611.04620}, the $*$-structure on the algebra corresponding to $(\cM_A, A)$ is given by 
$$
\begin{tikzpicture}[baseline=-.1cm, yscale=-1]
	\draw (-.3,-.6) arc (0:-180:.3cm) -- (-.9,1);
	\draw (.3,-1) -- (.3,1);
	\filldraw (.3,.6) circle (.05cm);
	\draw (-.3,0) .. controls ++(90:.5cm) and ++(-135:.5cm) .. (.3,.6);
	\roundNbox{unshaded}{(-.3,-.3)}{.3}{0}{0}{$f^*$}
	\node at (.5,-.8) {\scriptsize{$A$}};
	\node at (-.1,-.8) {\scriptsize{$a$}};
	\node at (-.5,.2) {\scriptsize{$A$}};
	\node at (.5,.8) {\scriptsize{$A$}};
	\node at (-1.1,-.3) {\scriptsize{$\overline{a}$}};
\end{tikzpicture}
=
\begin{tikzpicture}[baseline=-.1cm]
	\draw (-.3,.6) arc (0:180:.3cm) -- (-.9,-1);
	\draw (1.2,-1) -- (1.2,1.6);
	\filldraw (1.2,1.2) circle (.05cm);
	\draw (-.3,0) arc (-180:0:.4cm);
	\draw (.5,.6) .. controls ++(90:.5cm) and ++(-135:.5cm) .. (1.2,1.2);
	\roundNbox{unshaded}{(-.3,.3)}{.3}{0}{0}{$f^*$}
	\roundNbox{unshaded}{(.5,.3)}{.3}{.1}{.1}{$\sigma_A^{-1}$}
	\node at (-.1,.8) {\scriptsize{$a$}};
	\node at (-.5,-.2) {\scriptsize{$A$}};
	\node at (.7,-.2) {\scriptsize{$\overline{A}$}};
	\node at (1.4,1.4) {\scriptsize{$A$}};
	\node at (-1.1,.3) {\scriptsize{$\overline{a}$}};
\end{tikzpicture}
=
\begin{tikzpicture}[baseline=.4cm]
	\draw (-.3,-1) -- (-.3,.4);
	\draw (.3,-1) -- (.3,2);
	\filldraw (.3,1.6) circle (.05cm);
	\draw (-.3,1) .. controls ++(90:.5cm) and ++(-135:.5cm) .. (.3,1.6);
	\roundNbox{unshaded}{(-.3,-.3)}{.3}{0}{0}{$\overline{f}$}
	\roundNbox{unshaded}{(-.3,.7)}{.3}{.1}{.1}{$\sigma_A^{-1}$}
	\node at (.5,-.8) {\scriptsize{$A$}};
	\node at (-.5,-.8) {\scriptsize{$\overline{a}$}};
	\node at (-.5,.2) {\scriptsize{$A$}};
	\node at (.5,1.8) {\scriptsize{$A$}};
\end{tikzpicture}\,
$$
which is exactly the image of $j_a(f)$ under the isomorphism \eqref{eq:FreeModuleIsomorphism}.
We are now finished by the correspondence between cyclic $\cC$-module W*-categories and W*-algebra objects \cite[Thm.~3.24]{1611.04620}.
\end{proof}

Recall that we define $\bfL^2\bfA\in \Hilb(\cC)$ by $\bfL^2\bfA(a) = \bfA(a)$ with inner product \eqref{eq:L2AInnerProduct}.
Now $\bfL^2\bfA\in \Hilb(\cC)$ is compact, and thus there is an $H\in \cC$ such that $\bfL^2\bfA(a) = \cC(a, H)$ for all $a\in \cC$.

\begin{lem}
\label{lem:H=A}
The object $H\in \cC$ is canonically unitarily isomorphic to $A$.
\end{lem}
\begin{proof}
By injectivity of the Yoneda embedding $\cC \hookrightarrow \Hilb(\cC)$, it suffices to show that for all $a\in \cC$, the inner product \eqref{eq:CInnerProduct} on $\cC(a,A)$ 
and the inner product \eqref{eq:L2AInnerProduct} on 
$$
\cC(a, H)=\bfH(a)=\bfL^2\bfA(a)=\bfA(a) = \cC^{\natural}(a, A)
$$ 
agree.
Indeed, for $f,g\in \cC(a, A)$, by Lemma \ref{lem:UnitaryInvolution} and Remark \ref{rem:RealObject} together with \eqref{eq:BarStar}, we have
$$
\langle f | g\rangle_{\cC(a,A)}
=
\begin{tikzpicture}[baseline=-.1cm]
	\draw (0,.8) arc (180:0:.3cm) -- (.6,-.8) arc (0:-180:.3cm) -- (0,.8);
	\roundNbox{unshaded}{(0,.5)}{.3}{0}{0}{$f^*$}
	\roundNbox{unshaded}{(0,-.5)}{.3}{0}{0}{$g$}
	\node at (-.2,-1) {\scriptsize{$a$}};
	\node at (-.2,0) {\scriptsize{$A$}};
	\node at (-.2,1) {\scriptsize{$a$}};
	\node at (.8,0) {\scriptsize{$\overline{a}$}};
\end{tikzpicture}
=
\begin{tikzpicture}[baseline=-.1cm, xscale=-1]
	\draw (0,.8) arc (180:0:.3cm) -- (.6,-.8) arc (0:-180:.3cm) -- (0,.8);
	\roundNbox{unshaded}{(0,.5)}{.3}{0}{0}{$f^*$}
	\roundNbox{unshaded}{(0,-.5)}{.3}{0}{0}{$g$}
	\node at (-.2,-1) {\scriptsize{$a$}};
	\node at (-.2,0) {\scriptsize{$A$}};
	\node at (-.2,1) {\scriptsize{$a$}};
	\node at (.8,0) {\scriptsize{$\overline{a}$}};
\end{tikzpicture}
=
\begin{tikzpicture}[baseline = -.1cm, xscale=-1]
    \draw (-.5,-.3) arc (-180:0:.5cm);
    \draw (-.5,.3) arc (180:0:.5cm);
    \roundNbox{unshaded}{(-.5,0)}{.3}{0}{0}{$g$}
    \roundNbox{unshaded}{(.5,0)}{.3}{0}{0}{$\overline{f}$}
    \node at (-.65,-.5) {\scriptsize{$a$}};
    \node at (-.65,.5) {\scriptsize{$A$}};
    \node at (.65,-.5) {\scriptsize{$\overline{a}$}};
    \node at (.65,.5) {\scriptsize{$A$}};
\end{tikzpicture}
=
\begin{tikzpicture}[baseline = -.1cm, xscale=-1]
    \draw (-.5,-.3) arc (-180:0:.5cm);
    \draw (-.5,.3) arc (180:0:.5cm);
    \filldraw (0,.8) circle (.05cm);
    \filldraw (0,1.2) circle (.05cm) node [right] {\scriptsize{$i^*$}};
    \draw (0,.8) -- (0,1.2);
    \roundNbox{unshaded}{(-.5,0)}{.3}{0}{0}{$g$}
    \roundNbox{unshaded}{(.5,0)}{.3}{.2}{.2}{$j_a(f)$}
    \node at (-.65,-.5) {\scriptsize{$\mathbf{a}$}};
    \node at (-.65,.5) {\scriptsize{$\bfA$}};
    \node at (.65,-.5) {\scriptsize{$\overline{\mathbf{a}}$}};
    \node at (.65,.5) {\scriptsize{$\bfA$}};
    \node at (.15,1) {\scriptsize{$\bfA$}};
\end{tikzpicture}
=
\langle f|g\rangle_a
(i^*\circ i)
=
\langle f|g\rangle_a
$$
where $i^*\circ i = \id_{1_\cC}$ as in Definition \ref{defn:QSystem} since $A$ is normalized.
\end{proof}

\begin{rem}
\label{rem:ExplicitGNS}
By Lemma \ref{lem:H=A}, we see that the cyclic $\cC$-module W*-category $\cM_A$ with basepoint $A$ is equivalent to the cyclic $\cC$-module W*-subcategory of $\cC$ generated by $A$, which in turn is equivalent to the cyclic $\cC$-module W*-subcategory of $\Hilb(\cC)$ generated by $\bfL^2\bfA$, which corresponds to the W*-algebra object $\bfB(\bfL^2\bfA)$ by \cite[Ex.~3.35]{1611.04620}.
Thus \eqref{eq:ExplicitGNS} provides an explicit description of the canonical GNS faithful $*$-algebra natural transformation $\lambda: \bfA\Rightarrow \bfB(\bfL^2\bfA)$.
\end{rem}

\subsection{Morphisms}

Suppose $\theta\in\cC( A , B)$ is an involutive algebra morphism as in Section \ref{sec:Q-systemMorphisms}.
We now define a $*$-algebra natural transformation $\WStar(\theta): \WStar(A) \Rightarrow \WStar(B)$.
We use the notation $\WStar(A)=(\bfA, \mu^\bfA, i_\bfA, j^\bfA)$ and recall $\bfA(a) = \cC^{\natural}(a, A)$.
We have similar notation for $\WStar(B)$.

\begin{defn}
Given an involutive algebra morphism $\theta\in\cC( A , B)$, we define $\WStar(\theta): \bfA \Rightarrow \bfB$ by defining the $a$-component $\WStar(\theta)_a$ for $a\in\cC$ by 
$$
\cC^{\natural}(a, A)=\bfA(a)
\ni f
\mapsto
\theta \circ f 
\in 
\bfB(a)
=
\cC^{\natural}(a, B).
$$
Note that $\WStar(\theta)$ is automatically a natural transformation.
That $\theta$ is an algebra morphism implies $\WStar(\theta)$ is an algebra natural transformation, and involutivity of $\theta$ implies $\WStar(\theta)$ is a $*$-algebra natural transformation.
Indeed, for $f\in \bfA(a)$,
\begin{align*}
j^\bfB_a(\WStar(\theta)_a(f)) 
&=
j^\bfB_a(\theta \circ f)
=
\sigma_B^{-1} \circ \overline{\theta \circ f}
=
\sigma_B^{-1} \circ \overline{\theta} \circ \overline{f}
\\&=
\theta \circ \sigma_A^{-1} \circ \overline{f}
=
\theta \circ j_a^\bfA(f)
=
\WStar(\theta)_{\overline{a}}(j_a^\bfA(f)).
\end{align*}
\end{defn}

\begin{prop} 
$\bfW$\emph{*} is a functor.
\end{prop}
\begin{proof}
We must prove that $\WStar$ preserves identities and composites.
If $\theta =\id_A \in \cC(A,A)$, it is obvious that $\WStar(\theta) = \id_{\bfA}$.
Moreover, if we have $\theta_1\in \cC(A, B)$ and $\theta_2\in \cC(B, C)$, 
then for all $a\in\cC$ and $f\in \bfA(a)$, we have
$$
\WStar(\theta_1\circ \theta_2)_a(f)
=
\theta_1\circ \theta_2 \circ f
=
(\WStar(\theta_1)\circ \WStar(\theta_2))_a(f).
$$
We are finished.
\end{proof}

\section{From W*-algebra objects to Q-systems}
\label{sec:WtoQ}

\subsection{Objects}

Suppose $\bfA \in \Vec(\cC)$ is a compact connected W*-algebra object, and define $\bfH:=\bfL^2\bfA \in \Hilb(\cC)$, the Hilbert space object defined by $\bfH(a)=\bfL^2\bfA(a) = \bfA(a)$ with inner product \eqref{eq:L2AInnerProduct}.
Let $H\in \cC$ be the corresponding object satisfying $\bfH(a) = \cC(a,H)$ for all $a\in \cC$.
Recall $\bfA$ has a canonical state corresponding to the state on $\bfA(1_\cC) \cong \bbC$ via $i_A \mapsto 1_\bbC$.
The GNS representation gives a canonical faithful $*$-algebra natural transformation $\lambda: \bfA \Rightarrow \bfB(\bfL^2\bfA)$.

Now since $H \in \cC$, we get a canonical compact W*-algebra object $\bfH\otimes \overline{\bfH} : (\cC^\natural)^{\op} \to \Vec$ by $(\bfH\otimes \overline{\bfH})(a) = \cC(a, H\otimes \overline{H})$ as in Examples \ref{ex:StarAlgebraFromInnerHom} and \ref{ex:W*AlgebraFromInnerHom}, which is isomorphic to the compact W*-algebra object $\bfB(\bfH)\in \Vec(\cC)$.

\begin{defn}
\label{defn:InducedRepresentation}
Define $\pi: \bfA \Rightarrow \bfH\otimes \overline{\bfH}$ by 
$$
\bfA(a) \ni f \overset{\pi_a}{\longmapsto} 
d_H^{-1/2}
\begin{tikzpicture}[baseline=-.1cm]
	\draw (0,0) -- (0,.7);
	\draw (-.3,-.3) -- (-.3,-.7);
	\draw (.3,-.3) arc (-180:0:.3cm) -- (.9,.7);
	\roundNbox{unshaded}{(0,0)}{.3}{.3}{.3}{$\lambda_a(f)$}
	\node at (-.5,-.5) {\scriptsize{$\mathbf{a}$}};
	\node at (-.2,.5) {\scriptsize{$\bfH$}};
	\node at (1.1,0) {\scriptsize{$\overline{\bfH}$}};
\end{tikzpicture}
\in \Hom_{\Hilb(\cC)}(\mathbf{a},\bfH\otimes \overline{\bfH}) \cong \cC(a, H\otimes \overline{H})= (\bfH\otimes \overline{\bfH})(a).
$$
In particular, we have $\pi_{1_\cC}(i_\bfA) = d_H^{-1/2}\coev_H$ since $\lambda$ is unital.

It is easy to see that $\pi: \bfA \Rightarrow \bfH\otimes \overline{\bfH}$ is the faithful $*$-algebra natural transformation corresponding to $\lambda : \bfA \Rightarrow \bfB(\bfH)$ under the $*$-algebra natural isomorphism $\bfB(\bfH)\cong \bfH\otimes \overline{\bfH}$.
\end{defn}

\begin{nota}
For each $a\in \Irr(\cC)$, suppose we have $f,f'\in \bfA(a)$, and define $\alpha = \pi_a(f)$ and $\alpha'= \pi_a(f')$.
We then have the following identities from \eqref{eq:InnerProductInHilbC}:
\begin{equation}
\label{eq:IdentitiesForA(a)}
\begin{tikzpicture}[baseline=-.1cm]
	\draw (-.3,-.9) -- (-.3,-.6) arc (180:0:.3cm);
	\draw (-.3,.9) -- (-.3,.6) arc (-180:0:.3cm);
	\draw (0,-.3) -- (0,.3);
	\draw (.3,.6) arc (180:0:.3cm) -- (.9,-.6) arc (0:-180:.3cm);
	\filldraw[fill = white] (0,.3) circle (.05cm) node [above] {\scriptsize{$\alpha'$}};
	\filldraw[fill = white] (0,-.3) circle (.05cm) node [below] {\scriptsize{$\alpha^*$}};
	\node at (-.3,-1.1) {\scriptsize{$H$}};
	\node at (-.3,1.1) {\scriptsize{$H$}};
	\node at (.5,-.6) {\scriptsize{$\overline{H}$}};
	\node at (.5,.6) {\scriptsize{$\overline{H}$}};
	\node at (1.1,0) {\scriptsize{$H$}};
	\node at (-.2,0) {\scriptsize{$a$}};
\end{tikzpicture}
=
d_H^{-1}
\langle f | f'\rangle_a \id_H
\qquad
\begin{tikzpicture}[baseline=-.1cm]
	\draw (0,0) circle (.4cm);
	\draw (0,-.4) arc (-180:0:.6cm) -- (1.2,.4) arc (0:180:.6cm);
	\filldraw[fill = white] (0,.4) circle (.05cm) node [below] {\scriptsize{$\alpha^*$}};
	\filldraw[fill = white] (0,-.4) circle (.05cm) node [above] {\scriptsize{$\alpha'$}};
	\node at (-.6,0) {\scriptsize{$H$}};
	\node at (.6,0) {\scriptsize{$\overline{H}$}};
	\node at (-.2,.6) {\scriptsize{$a$}};
	\node at (-.2,-.6) {\scriptsize{$a$}};
	\node at (1,0) {\scriptsize{$\overline{a}$}};
\end{tikzpicture}
=
\langle f|f'\rangle_a
\qquad
\begin{tikzpicture}[baseline=-.1cm]
	\draw (0,0) circle (.4cm);
	\draw (0,-.4)-- (0,-.8);
	\draw (0,.4)-- (0,.8);
	\filldraw[fill = white] (0,.4) circle (.05cm) node [below] {\scriptsize{$\alpha^*$}};
	\filldraw[fill = white] (0,-.4) circle (.05cm) node [above] {\scriptsize{$\alpha'$}};
	\node at (-.6,0) {\scriptsize{$H$}};
	\node at (.6,0) {\scriptsize{$\overline{H}$}};
	\node at (-.2,.6) {\scriptsize{$a$}};
	\node at (-.2,-.6) {\scriptsize{$a$}};
\end{tikzpicture}
=
d_a^{-1} \langle f|f'\rangle_a \id_a
\end{equation}
For each $a\in \Irr(\cC)$, pick an orthonormal basis $\ONB(\bfA(a))$ of $\bfA(a)$ under the inner product \eqref{eq:L2AInnerProduct}, and define
$$
\cB_a^\bfA
=
\set{d_a^{1/2}\pi_a(f)}{ f\in \ONB(\bfA(a))}.
$$
We will suppress the superscript and merely write $\cB_a$ when no confusion can arise.
By convention, we pick $\ONB(\bfA(a)) = \{i_\bfA\}$, and thus $\cB_{1_\cC} = \{ d_H^{-1/2} \coev_H\}$.
This normalization means that for $\alpha, \alpha' \in \cB_a$, we have 
\begin{equation}
\label{eqn:OpenInnerProduct}
\begin{tikzpicture}[baseline=-.1cm]
	\draw (0,0) circle (.4cm);
	\draw (0,-.4)-- (0,-.8);
	\draw (0,.4)-- (0,.8);
	\filldraw[fill = white] (0,.4) circle (.05cm) node [below] {\scriptsize{$\alpha^*$}};
	\filldraw[fill = white] (0,-.4) circle (.05cm) node [above] {\scriptsize{$\alpha'$}};
	\node at (-.6,0) {\scriptsize{$H$}};
	\node at (.6,0) {\scriptsize{$\overline{H}$}};
	\node at (-.2,.6) {\scriptsize{$a$}};
	\node at (-.2,-.6) {\scriptsize{$a$}};
\end{tikzpicture}
=
\delta_{\alpha=\alpha'} \id_a.
\end{equation}
Thus $\cB_a$ is an orthonormal basis for $\pi_a(\bfA(a))$ under the new inner product \eqref{eqn:OpenInnerProduct}.
\end{nota}

\begin{defn}
We define the following distinguished element of $\cC(H\otimes \overline{H}, H\otimes \overline{H})$:
$$
p = 
\sum_{a\in \Irr(\cC)} \sum_{\alpha \in \cB_a}
\begin{tikzpicture}[baseline=-.1cm]
	\draw (-.3,-.6) arc (180:0:.3cm);
	\draw (-.3,.6) arc (-180:0:.3cm);
	\draw (0,-.3) -- (0,.3);
	\filldraw[fill = white] (0,.3) circle (.05cm) node [above] {\scriptsize{$\alpha$}};
	\filldraw[fill = white] (0,-.3) circle (.05cm) node [below] {\scriptsize{$\alpha^*$}};
	\node at (-.3,-.8) {\scriptsize{$H$}};
	\node at (-.3,.8) {\scriptsize{$H$}};
	\node at (.3,-.8) {\scriptsize{$\overline{H}$}};
	\node at (.3,.8) {\scriptsize{$\overline{H}$}};
	\node at (-.2,0) {\scriptsize{$a$}};
\end{tikzpicture}
.
$$
Note that the element $p$ is independent of the choices of orthonormal bases $\cB_a$ for $a\in \Irr(\cC)$ under the inner product \eqref{eqn:OpenInnerProduct}.
\end{defn}

\begin{lem}
\label{lem:SymmetricallySelfDualProjector}
The element $p$ is a symmetrically self-dual orthogonal projector.
That is, $p^2=p^*=p$, and 
$
(p \otimes \id_{H\otimes \overline{H}})\circ \coev_{H\otimes \overline{H}}
=
(\id_{H\otimes \overline{H}}\otimes p) \circ \coev_{H\otimes \overline{H}}.
$
\end{lem}
\begin{proof}
It is obvious that $p^*=p$, and it follows readily from \eqref{eqn:OpenInnerProduct} that $p^2=p$.
To verify the final condition, for each $a\in \Irr(\cC)$, there is a unique $a'\in \Irr(\cC)$ with $a'\cong \overline{a}$.
We pick a unitary isomorphism $\gamma_a \in \cC(a',\overline{a})$, and we see that
$$
(p \otimes \id_{H\otimes \overline{H}})\circ \coev_{H\otimes \overline{H}}
=
\sum_{a\in \Irr(\cC)} \sum_{\alpha \in \cB_a}
\begin{tikzpicture}[baseline=.3cm]
	\draw (-.9,.6) arc (-180:0:.3cm);
	\draw (.3,.6) arc (-180:0:.3cm);
	\draw (-.6,.3) arc (-180:0:.6cm);
	\filldraw[fill = white] (-.6,.3) circle (.05cm) node [above] {\scriptsize{$\alpha$}};
	\filldraw[fill = white] (.6,.3) circle (.05cm) node [above] {\scriptsize{$\overline{\alpha}$}};
	\node at (-.9,.8) {\scriptsize{$H$}};
	\node at (-.3,.8) {\scriptsize{$\overline{H}$}};
	\node at (.9,.8) {\scriptsize{$\overline{H}$}};
	\node at (.3,.8) {\scriptsize{$H$}};
	\node at (-.8,.1) {\scriptsize{$a$}};
	\node at (.8,.1) {\scriptsize{$\overline{a}$}};
\end{tikzpicture}
=
\sum_{a\in \Irr(\cC)} \sum_{\alpha \in \cB_{a}}
\begin{tikzpicture}[baseline=-.1cm]
	\draw (-.9,.6) arc (-180:0:.3cm);
	\draw (.3,.6) arc (-180:0:.3cm);
	\draw (-.6,.3) -- (-.6,-.55) arc (-180:0:.6cm) -- (.6,.3);
	\filldraw[fill = white] (.6,.3) circle (.05cm) node [above] {\scriptsize{$\overline{\alpha}$}};
	\filldraw[fill = white] (-.6,.3) circle (.05cm) node [above] {\scriptsize{$\alpha$}};
	\roundNbox{unshaded}{(-.6,-.3)}{.25}{0}{0}{$\overline{\gamma_a}$}
	\roundNbox{unshaded}{(.6,-.3)}{.25}{0}{0}{$\gamma_a$}
	\node at (-.9,.8) {\scriptsize{$H$}};
	\node at (-.3,.8) {\scriptsize{$\overline{H}$}};
	\node at (.9,.8) {\scriptsize{$\overline{H}$}};
	\node at (.3,.8) {\scriptsize{$H$}};
	\node at (-.8,.1) {\scriptsize{$a$}};
	\node at (.8,.1) {\scriptsize{$\overline{a}$}};
	\node at (-.8,.-.8) {\scriptsize{$\overline{a'}$}};
	\node at (.8,-.8) {\scriptsize{$a'$}};
\end{tikzpicture}
$$
Now since $\set{\overline{\alpha} \circ \gamma_a}{\alpha \in \cB_a}$ is an orthonormal basis for $\pi_{a'}(\bfA(a'))$ under the inner product \eqref{eqn:OpenInnerProduct}, and since $p$ is independent of the choice of orthonormal bases $\cB_a$, we see that the right hand side above is equal to
\begin{equation*}
\sum_{a\in \Irr(\cC)} \sum_{\alpha \in \cB_a}
\begin{tikzpicture}[baseline=.3cm]
	\draw (-.9,.6) arc (-180:0:.3cm);
	\draw (.3,.6) arc (-180:0:.3cm);
	\draw (-.6,.3) arc (-180:0:.6cm);
	\filldraw[fill = white] (.6,.3) circle (.05cm) node [above] {\scriptsize{$\alpha$}};
	\filldraw[fill = white] (-.6,.3) circle (.05cm) node [above] {\scriptsize{$\overline{\alpha}$}};
	\node at (-.9,.8) {\scriptsize{$H$}};
	\node at (-.3,.8) {\scriptsize{$\overline{H}$}};
	\node at (.9,.8) {\scriptsize{$\overline{H}$}};
	\node at (.3,.8) {\scriptsize{$H$}};
	\node at (.8,.1) {\scriptsize{$a$}};
	\node at (-.8,.1) {\scriptsize{$\overline{a}$}};
\end{tikzpicture}
=
(\id_{H\otimes \overline{H}}\otimes p) \circ \coev_{H\otimes \overline{H}}.
\qedhere
\end{equation*}
\end{proof}

We now define important structure constants which are essentially certain 6j symbols.

\begin{defn}
For $a,b,c\in \Irr(\cC)$, $\alpha \in \cB_a$, $\beta \in \cB_b$, $\gamma \in  \cB_c$, and $\delta \in \cC(c, a\otimes b)$, we define the
\emph{tetrahedral structure constant}
$$
\Delta
\left(
\begin{array}{ccc} a & b & c \\ \alpha & \beta & \gamma \end{array}
\middle| 
\begin{array}{c} \delta \end{array}
\right)
=
\begin{tikzpicture}[baseline=-.1cm]
	\coordinate (a) at (-.5,0);
	\coordinate (b) at (.5,0);
	\coordinate (c) at (0,1);
	\coordinate (c') at (1,1.3);
	\coordinate (d) at (0,-1);
	\coordinate (d') at (1,-1.3);
	\draw (c) -- ($ (c) + (0,.3) $) .. controls ++(90:.5cm) and ++(90:.5cm) .. (c') -- node [right] {\scriptsize{$\overline{c}$}} (d') .. controls ++(270:.5cm) and ++(270:.5cm) .. ($ (d) + (0,-.3) $) -- (d);
	\draw (a) .. controls ++(45:.5cm) and ++(135:.5cm) ..  node [above] {\scriptsize{$H$}} (b);
	\draw (a) to node [left] {\scriptsize{$a$}} (d);
	\draw (a) to node [left] {\scriptsize{$H$}} (c);
	\draw (b) to node [right] {\scriptsize{$b$}} (d);
	\draw (b) to node [right] {\scriptsize{$\overline{H}$}} (c);
	\ncircle{unshaded}{(-.5,0)}{.3}{$\alpha$}
	\ncircle{unshaded}{(.5,0)}{.3}{$\beta$}
	\ncircle{unshaded}{(0,-1)}{.3}{$\delta$}
	\ncircle{unshaded}{(0,1)}{.3}{$\gamma^*$}
	\node at (-.2,1.5) {\scriptsize{$c$}};
	\node at (-.2,-1.5) {\scriptsize{$c$}};
\end{tikzpicture}.
$$
Note that these structure constants have a $\bbZ/3\bbZ$ symmetry:
\begin{equation}
\label{eq:Z3Symmetry}
\Delta\left(
\begin{array}{ccc} a & b & c \\ \alpha & \beta & \gamma \end{array}
\middle| 
\begin{array}{c} \delta \end{array}
\right)
=
\Delta\left(
\begin{array}{ccc} \overline{c} & a & \overline{b} \\  \overline{\gamma}& \alpha & \overline{\beta} \end{array}
\middle| 
\begin{array}{c} \cF(\delta) \end{array}
\right)
=
\Delta\left(
\begin{array}{ccc} b & \overline{c} & \overline{a} \\  \beta & \overline{\gamma} & \overline{\alpha} \end{array}
\middle| 
\begin{array}{c} \cF^2(\delta) \end{array}
\right).
\end{equation}
\end{defn}

\begin{ex}
Using the tetrahedral structure constants, we get the following two important identities.
Both are proved by expanding the left hand sides out using the respective inner products.
\begin{enumerate}[(1)]
\item
For all $a,b,c\in \Irr(\cC)$ and $\alpha\in \cB_a$ and $\beta \in \cB_b$, and $\delta\in \Isom(c,a\otimes b)$, 
expanding using the inner product \eqref{eqn:OpenInnerProduct} on $\cB_c$ yields
\begin{equation}
\label{eq:ReplaceABD}
\begin{tikzpicture}[baseline=-.5cm]
	\coordinate (a) at (-.5,0);
	\coordinate (b) at (.5,0);
	\coordinate (d) at (0,-1);
	\draw ($ (a) + (0,.6) $) -- (a);
	\draw ($ (b) + (0,.6) $) -- (b);
	\draw ($ (d) + (0,-.6) $) -- (d);
	\draw (a) .. controls ++(45:.5cm) and ++(135:.5cm) ..  node [above] {\scriptsize{$H$}} (b);
	\draw (a) to node [left] {\scriptsize{$a$}} (d);
	\draw (b) to node [right] {\scriptsize{$b$}} (d);
	\ncircle{unshaded}{(-.5,0)}{.3}{$\alpha$}
	\ncircle{unshaded}{(.5,0)}{.3}{$\beta$}
	\ncircle{unshaded}{(0,-1)}{.3}{$\delta$}
	\node at (-.2,-1.5) {\scriptsize{$c$}};
	\node at (-.7,.5) {\scriptsize{$H$}};
	\node at (.7,.5) {\scriptsize{$\overline{H}$}};
\end{tikzpicture}
=
d_c^{-1}
\sum_{\gamma \in \cB_c}
\Delta\left(
\begin{array}{ccc} a & b & c \\ \alpha & \beta & \gamma \end{array}
\middle| 
\begin{array}{c} \delta \end{array}
\right)
\begin{tikzpicture}[baseline=-.1cm]
	\draw (-.3,.3) arc (-180:0:.3cm);
	\draw (0,0) -- (0,-.3);
	\filldraw[fill = white] (0,0) circle (.05cm) node [above] {\scriptsize{$\gamma$}};
	\node at (-.3,.5) {\scriptsize{$H$}};
	\node at (.3,.5) {\scriptsize{$\overline{H}$}};
	\node at (0,-.5) {\scriptsize{$c$}};
\end{tikzpicture}.
\end{equation}
\item
For all $a,b,c\in \Irr(\cC)$, $\alpha \in \cB_a$, $\beta \in \cB_b$, and $\gamma\in \cB_c$, 
expanding using the inner product on $\Isom(c, a\otimes b)$ yields
\begin{equation}
\label{eq:ReplaceABC}
\begin{tikzpicture}[baseline=.5cm]
	\coordinate (a) at (-.5,0);
	\coordinate (b) at (.5,0);
	\coordinate (c) at (0,1);
	\draw (a) -- ($ (a) + (0,-.6) $);
	\draw (b) -- ($ (b) + (0,-.6) $);
	\draw (c) -- ($ (c) + (0,.6) $);
	\draw (a) .. controls ++(45:.5cm) and ++(135:.5cm) ..  node [above] {\scriptsize{$H$}} (b);
	\draw (a) to node [left] {\scriptsize{$H$}} (c);
	\draw (b) to node [right] {\scriptsize{$\overline{H}$}} (c);
	\ncircle{unshaded}{(-.5,0)}{.3}{$\alpha$}
	\ncircle{unshaded}{(.5,0)}{.3}{$\beta$}
	\ncircle{unshaded}{(0,1)}{.3}{$\gamma^*$}
	\node at (-.2,1.5) {\scriptsize{$c$}};
	\node at (-.7,-.5) {\scriptsize{$a$}};
	\node at (.7,-.5) {\scriptsize{$b$}};
\end{tikzpicture}
=
d_c^{-1}
\sum_{\delta \in \Isom(c, a\otimes b)}
\Delta\left(
\begin{array}{ccc} a & b & c \\ \alpha & \beta & \gamma \end{array}
\middle| 
\begin{array}{c} \delta \end{array}
\right)
\begin{tikzpicture}[baseline=-.1cm]
	\draw (-.3,-.3) arc (180:0:.3cm);
	\draw (0,0) -- (0,.3);
	\filldraw[fill = white] (0,0) circle (.05cm) node [below] {\scriptsize{$\delta^*$}};
	\node at (-.3,-.5) {\scriptsize{$a$}};
	\node at (.3,-.5) {\scriptsize{$b$}};
	\node at (0,.5) {\scriptsize{$c$}};
\end{tikzpicture}.
\end{equation}
\end{enumerate}
\end{ex}

\begin{lem}
\label{lem:RemoveTopP}
The following two morphisms in $\cC(H\otimes \overline{H}\otimes H \otimes \overline{H}, H\otimes \overline{H})$ are equal:
$$
\begin{tikzpicture}[baseline=-.1cm]
	\draw (-.4,-.7) -- (-.4,.4);
	\draw (.4,-.7) -- (.4,.4);
	\draw (-.2,-.7) -- (-.2,-.1) arc (180:0:.2cm) -- (.2,-.7);
	\roundNbox{unshaded}{(-.3,-.3)}{.2}{0}{0}{$p$};
	\roundNbox{unshaded}{(.3,-.3)}{.2}{0}{0}{$p$};
\end{tikzpicture}
=
\begin{tikzpicture}[baseline=0cm]
	\draw (-.4,-.7) -- (-.4,-.1) .. controls ++(90:.1cm) and ++(270:.2cm) .. (-.1,.3) -- (-.1,.9);
	\draw (.4,-.7) -- (.4,-.1) .. controls ++(90:.1cm) and ++(270:.2cm) .. (.1,.3) -- (.1,.9);
	\draw (-.2,-.7) -- (-.2,-.1) .. controls ++(90:.2cm) and ++(90:.2cm) .. (.2,-.1) -- (.2,-.7);
	\roundNbox{unshaded}{(-.3,-.3)}{.2}{0}{0}{$p$};
	\roundNbox{unshaded}{(.3,-.3)}{.2}{0}{0}{$p$};
	\roundNbox{unshaded}{(0,.5)}{.2}{0}{0}{$p$};
\end{tikzpicture}
.
$$
\end{lem}
\begin{proof}
Using \eqref{eq:ReplaceABD} and \eqref{eq:ReplaceABC} and the fusion relation 
$$
\begin{tikzpicture}[baseline=-.1cm]
	\draw (-.1,-.3) -- (-.1,.3);
	\draw (.1,-.3) -- (.1,.3);
	\node at (-.1,-.5) {\scriptsize{$a$}};
	\node at (.1,-.5) {\scriptsize{$b$}};
\end{tikzpicture}
=
\sum_{c\in\Irr(\cC)}
\sum_{\delta\in \Isom(c, a\otimes b)} 
\begin{tikzpicture}[baseline=-.1cm]
	\draw (-.3,-.6) arc (180:0:.3cm);
	\draw (-.3,.6) arc (-180:0:.3cm);
	\draw (0,-.3) -- (0,.3);
	\filldraw[fill = white] (0,.3) circle (.05cm) node [above] {\scriptsize{$\delta$}};
	\filldraw[fill = white] (0,-.3) circle (.05cm) node [below] {\scriptsize{$\delta^*$}};
	\node at (-.3,-.8) {\scriptsize{$a$}};
	\node at (-.3,.8) {\scriptsize{$a$}};
	\node at (.3,-.8) {\scriptsize{$b$}};
	\node at (.3,.8) {\scriptsize{$b$}};
	\node at (-.2,0) {\scriptsize{$c$}};
\end{tikzpicture}
,
$$
it is straightforward to show that both morphisms are equal to
\begin{align}
&\sum_{a,b,c\in \Irr(\cC)} 
d_c^{-1}
\sum_{\substack{
\alpha\in \cB_a
\\
\beta\in\cB_b
\\
\gamma\in\cB_c
}}
\sum_{\delta\in \Isom(c, a\otimes b)}
\Delta\left(
\begin{array}{ccc} a & b & c \\ \alpha & \beta & \gamma \end{array}
\middle| 
\begin{array}{c} \delta \end{array}
\right)
\begin{tikzpicture}[baseline=.3cm]
	\draw (-.9,-.6) arc (180:0:.3cm);
	\draw (.3,-.6) arc (180:0:.3cm);
	\draw (-.6,-.3) arc (180:0:.6cm);
	\draw (0,.3) -- (0,.9);
	\draw (-.3,1.2) arc (-180:0:.3cm);
	\filldraw[fill = white] (-.6,-.3) circle (.05cm) node [below] {\scriptsize{$\alpha^*$}};
	\filldraw[fill = white] (.6,-.3) circle (.05cm) node [below] {\scriptsize{$\beta^*$}};
	\filldraw[fill = white] (0,.3) circle (.05cm) node [below] {\scriptsize{$\delta^*$}};
	\filldraw[fill = white] (0,.9) circle (.05cm) node [above] {\scriptsize{$\gamma$}};
	\node at (-.9,-.8) {\scriptsize{$H$}};
	\node at (-.3,-.8) {\scriptsize{$\overline{H}$}};
	\node at (.9,-.8) {\scriptsize{$\overline{H}$}};
	\node at (.3,-.8) {\scriptsize{$H$}};
	\node at (.3,1.4) {\scriptsize{$\overline{H}$}};
	\node at (-.3,1.4) {\scriptsize{$H$}};
	\node at (-.8,-.1) {\scriptsize{$a$}};
	\node at (.8,-.1) {\scriptsize{$b$}};
	\node at (.2,.6) {\scriptsize{$c$}};
\end{tikzpicture}
\notag\\&=
\sum_{a,b,c\in \Irr(\cC)} 
\sum_{\substack{
\alpha\in \cB_a
\\
\beta\in\cB_b
\\
\gamma\in\cB_c
}}
\sum_{\delta\in \ONB(c, a\otimes b)}
\Delta\left(
\begin{array}{ccc} a & b & c \\ \alpha & \beta & \gamma \end{array}
\middle| 
\begin{array}{c} \delta \end{array}
\right)
\begin{tikzpicture}[baseline=.3cm]
	\draw (-.9,-.6) arc (180:0:.3cm);
	\draw (.3,-.6) arc (180:0:.3cm);
	\draw (-.6,-.3) arc (180:0:.6cm);
	\draw (0,.3) -- (0,.9);
	\draw (-.3,1.2) arc (-180:0:.3cm);
	\filldraw[fill = white] (-.6,-.3) circle (.05cm) node [below] {\scriptsize{$\alpha^*$}};
	\filldraw[fill = white] (.6,-.3) circle (.05cm) node [below] {\scriptsize{$\beta^*$}};
	\filldraw[fill = white] (0,.3) circle (.05cm) node [below] {\scriptsize{$\delta^*$}};
	\filldraw[fill = white] (0,.9) circle (.05cm) node [above] {\scriptsize{$\gamma$}};
	\node at (-.9,-.8) {\scriptsize{$H$}};
	\node at (-.3,-.8) {\scriptsize{$\overline{H}$}};
	\node at (.9,-.8) {\scriptsize{$\overline{H}$}};
	\node at (.3,-.8) {\scriptsize{$H$}};
	\node at (.3,1.4) {\scriptsize{$\overline{H}$}};
	\node at (-.3,1.4) {\scriptsize{$H$}};
	\node at (-.8,-.1) {\scriptsize{$a$}};
	\node at (.8,-.1) {\scriptsize{$b$}};
	\node at (.2,.6) {\scriptsize{$c$}};
\end{tikzpicture}
\label{eq:SubalgebraMultiplication}
\end{align}
as $\Isom(c, a\otimes b) = \sqrt{d_c}\ONB(c, a\otimes b)$.
\end{proof}

\begin{rem}
Notice that the the morphism \eqref{eq:SubalgebraMultiplication} is invariant under cyclic rotation, since the sum is independent of the choice of $\ONB(c, a\otimes b)$, and the 1-click rotation $\cF$ produces an orthonormal basis of $\cC(\overline{b}, \overline{c}\otimes a)$ under the standard inner product \eqref{eq:CInnerProduct} which caps off all strings.
Finally, we use the $\bbZ/3\bbZ$ symmetry of the tetrahedral structure constants \eqref{eq:Z3Symmetry}.
\end{rem}

\begin{cor}
\label{cor:RemoveAnyP}
The morphism
$
\begin{tikzpicture}[baseline=0cm]
	\draw (-.4,-.7) -- (-.4,-.1) .. controls ++(90:.1cm) and ++(270:.2cm) .. (-.1,.3) -- (-.1,.9);
	\draw (.4,-.7) -- (.4,-.1) .. controls ++(90:.1cm) and ++(270:.2cm) .. (.1,.3) -- (.1,.9);
	\draw (-.2,-.7) -- (-.2,-.1) .. controls ++(90:.2cm) and ++(90:.2cm) .. (.2,-.1) -- (.2,-.7);
	\roundNbox{unshaded}{(-.3,-.3)}{.2}{0}{0}{$p$};
	\roundNbox{unshaded}{(.3,-.3)}{.2}{0}{0}{$p$};
	\roundNbox{unshaded}{(0,.5)}{.2}{0}{0}{$p$};
\end{tikzpicture}
$
remains unchanged by removing any single $p$.
\end{cor}
\begin{proof}
Immediate from Lemmas \ref{lem:SymmetricallySelfDualProjector} and \ref{lem:RemoveTopP} using 2-click rotations.
\end{proof}

\begin{lem}
\label{lem:CapOffP}
Capping off $p$ produces $\id_H$, i.e.,
$
\begin{tikzpicture}[baseline=-.1cm]
	\draw (-.1,-.4) -- (-.1,.4);
	\draw (.1,.2) arc (180:0:.1cm) -- (.3,-.2) arc (0:-180:.1cm);
	\roundNbox{unshaded}{(0,0)}{.2}{0}{0}{$p$};
\end{tikzpicture}
=
\id_H
$.
\end{lem}
\begin{proof}
Using the definition of $\cB_a$ and applying \eqref{eq:IdentitiesForA(a)}, we have 
\begin{equation*}
\begin{tikzpicture}[baseline=-.1cm]
	\draw (-.1,-.4) -- (-.1,.4);
	\draw (.1,.2) arc (180:0:.1cm) -- (.3,-.2) arc (0:-180:.1cm);
	\roundNbox{unshaded}{(0,0)}{.2}{0}{0}{$p$};
\end{tikzpicture}
=
\sum_{a\in \Irr(\cC)}
\sum_{\alpha\in \cB_a}
\begin{tikzpicture}[baseline=-.1cm]
	\draw (-.3,-.9) -- (-.3,-.6) arc (180:0:.3cm);
	\draw (-.3,.9) -- (-.3,.6) arc (-180:0:.3cm);
	\draw (0,-.3) -- (0,.3);
	\draw (.3,.6) arc (180:0:.3cm) -- (.9,-.6) arc (0:-180:.3cm);
	\filldraw[fill = white] (0,.3) circle (.05cm) node [above] {\scriptsize{$\alpha'$}};
	\filldraw[fill = white] (0,-.3) circle (.05cm) node [below] {\scriptsize{$\alpha^*$}};
	\node at (-.3,-1.1) {\scriptsize{$H$}};
	\node at (-.3,1.1) {\scriptsize{$H$}};
	\node at (.5,-.6) {\scriptsize{$\overline{H}$}};
	\node at (.5,.6) {\scriptsize{$\overline{H}$}};
	\node at (1.1,0) {\scriptsize{$H$}};
	\node at (-.2,0) {\scriptsize{$a$}};
\end{tikzpicture}
=
\sum_{a\in \Irr(\cC)}
\dim(\bfA(a)) \frac{d_a}{d_H} \id_H
=
\id_H.
\qedhere
\end{equation*}
\end{proof}

A version of the following proposition appears in \cite{MR3308880}.

\begin{prop}
\label{prop:W*ToFrobenius}
The object $\im(p)\subset H\otimes \overline{H}$ with multiplication and unit morphisms
$$
m=
d_H^{1/2}
\begin{tikzpicture}[baseline=0cm]
	\draw (-.4,-.7) -- (-.4,-.1) .. controls ++(90:.1cm) and ++(270:.2cm) .. (-.1,.3) -- (-.1,.9);
	\draw (.4,-.7) -- (.4,-.1) .. controls ++(90:.1cm) and ++(270:.2cm) .. (.1,.3) -- (.1,.9);
	\draw (-.2,-.7) -- (-.2,-.1) .. controls ++(90:.2cm) and ++(90:.2cm) .. (.2,-.1) -- (.2,-.7);
	\roundNbox{unshaded}{(-.3,-.3)}{.2}{0}{0}{$p$};
	\roundNbox{unshaded}{(.3,-.3)}{.2}{0}{0}{$p$};
	\roundNbox{unshaded}{(0,.5)}{.2}{0}{0}{$p$};
\end{tikzpicture}
\qquad\qquad
i=
d_H^{-1/2}
\begin{tikzpicture}[baseline=-.1cm]
	\draw (-.1,0) -- (-.1,.5);
	\draw (.1,0) -- (.1,.5);
	\draw (-.1,-.2) -- (-.1,-.4) arc (-180:0:.1cm) -- (.1,-.2);
	\roundNbox{unshaded}{(0,0)}{.2}{0}{0}{$p$};
\end{tikzpicture}
$$
is a normalized irreducible Q-system in $\cC$.
\end{prop}
\begin{proof}
A straightforward calculation using Corollary \ref{cor:RemoveAnyP} shows that $(\im(p),m,i)$ is a separable C* Frobenius algebra.
(See \cite[Lem.~5.8]{1507.04794} for a hint.)
Using Lemma \ref{lem:CapOffP}, we calculate that $m\circ m^* = d_H\id_H$ and $i\circ i^* = \id_{1_\cC}$, and thus $\im(p)$ is a normalized Q-system.
It is obvious that $\im(p)$ is connected as $\dim(\cB_{1_\cC} )=1$, and thus the Q-system is irreducible.
\end{proof}

\begin{lem}
\label{lem:RemovePFromAMorphisms}
For all $a\in \cC$ and $g\in \bfA(a)$, $p\circ \pi_a(g) = \pi_a(g)$.
In particular, the unit for the Q-system in Proposition \ref{prop:W*ToFrobenius} is given by
$i = d_H^{-1/2} (p\circ \coev_H) = d_H^{-1/2} \coev_H$.
\end{lem}
\begin{proof}
For $a\in \cC$ and $g\in \bfA(a)$, we have
\begin{equation}
\label{eq:p circ pi(g)}
p\circ \pi_a(g)
=
d_H^{-1/2}
\sum_{c\in\Irr(\cC)}
\sum_{\alpha\in \cB_c}
\begin{tikzpicture}[baseline=.4cm]
	\draw (-.3,-.6) -- (-.3,0);
	\draw (.3,-.3) arc (-180:0:.3cm) -- (.9,.55) arc (0:180:.45cm) -- (0,0);
	\draw (.45,1) -- (.45,1.4);
	\draw (0,1.85) arc (-180:0:.45cm);
	\filldraw[fill=white] (.45,1.4) circle (.05cm) node [above] {\scriptsize{$\alpha$}};
	\filldraw[fill=white] (.45,1) circle (.05cm) node [below] {\scriptsize{$\alpha^*$}};
	\roundNbox{unshaded}{(0,0)}{.3}{.2}{.2}{$\lambda_a(g)$}
	\node at (.25,1.2) {\scriptsize{$c$}};
	\node at (-.5,-.5) {\scriptsize{$a$}};
	\node at (-.2,.5) {\scriptsize{$H$}};
	\node at (1.1,.5) {\scriptsize{$\overline{H}$}};
	\node at (-.2,1.75) {\scriptsize{$H$}};
	\node at (1.1,1.75) {\scriptsize{$\overline{H}$}};
\end{tikzpicture}
=
d_H^{-3/2}
\sum_{c\in\Irr(\cC)}
\sum_{f\in \ONB(\bfA(c))}
d_c
\begin{tikzpicture}[baseline=-.1cm]
	\draw (.3,-1.3) arc (-180:0:.3cm) -- (.9,.3) arc (0:180:.3cm);
	\draw (.3,1.1) arc (-180:0:.3cm) -- (.9,2);
	\draw (-.3,0) -- (-.3,1.4);
	\draw (0,1.4) --(0,2);
	\draw (-.3,-1) -- (-.3,-1.8);
	\draw (0,-1) -- (0,0);
	\roundNbox{unshaded}{(0,1.4)}{.3}{.2}{.2}{$\lambda_c(f)$}
	\roundNbox{unshaded}{(0,0)}{.3}{.3}{.3}{$\lambda_c(f)^*$}
	\roundNbox{unshaded}{(0,-1)}{.3}{.2}{.2}{$\lambda_a(g)$}
	\node at (-.5,.7) {\scriptsize{$c$}};
	\node at (-.5,-1.7) {\scriptsize{$a$}};
	\node at (-.2,1.9) {\scriptsize{$H$}};
	\node at (-.2,-.5) {\scriptsize{$H$}};
	\node at (1.1,1.9) {\scriptsize{$\overline{H}$}};
	\node at (1.1,-.5) {\scriptsize{$\overline{H}$}};
\end{tikzpicture}
\end{equation}
where $\ONB(\bfA(c))$ is with respect to inner product \eqref{eq:L2AInnerProduct}.
For $c\in\Irr(\cC)$ and $f\in \ONB(\bfA(c))$, we take inner products against an orthonormal basis $\Isom(c,a)$ for inner product \eqref{eq:IsomInnerProduct} on $\cC(c,a)$ to see
\begin{align*}
\begin{tikzpicture}[baseline=-.6cm]
	\draw (.3,-1.3) arc (-180:0:.3cm) -- (.9,.3) arc (0:180:.3cm);
	\draw (-.3,0) -- (-.3,.8);
	\draw (-.3,-1) -- (-.3,-1.8);
	\draw (0,-1) -- (0,0);
	\roundNbox{unshaded}{(0,0)}{.3}{.3}{.3}{$\lambda_c(f)^*$}
	\roundNbox{unshaded}{(0,-1)}{.3}{.2}{.2}{$\lambda_a(g)$}
	\node at (-.5,.7) {\scriptsize{$c$}};
	\node at (-.5,-1.7) {\scriptsize{$a$}};
	\node at (-.2,-.5) {\scriptsize{$H$}};
	\node at (1.1,-.5) {\scriptsize{$\overline{H}$}};
\end{tikzpicture}
=
\begin{tikzpicture}[baseline=-1.6cm]
	\draw (.3,-1.3) arc (-180:0:.3cm) -- (.9,.3) arc (0:180:.3cm);
	\draw (-.3,0) -- (-.3,.8);
	\draw (-.3,-1) -- (-.3,-3.6);
	\draw (0,-1) -- (0,0);
	\roundNbox{unshaded}{(0,0)}{.3}{.3}{.3}{$\lambda_c(f)^*$}
	\roundNbox{unshaded}{(0,-1)}{.3}{.2}{.2}{$\lambda_a(g)$}
	\roundNbox{unshaded}{(-.3,-2)}{.3}{0}{0}{$\beta$}
	\roundNbox{unshaded}{(-.3,-3)}{.3}{0}{0}{$\beta^*$}
	\node at (-.5,.7) {\scriptsize{$c$}};
	\node at (-.5,-1.5) {\scriptsize{$a$}};
	\node at (-.5,-2.5) {\scriptsize{$c$}};
	\node at (-.5,-3.5) {\scriptsize{$a$}};
	\node at (-.2,-.5) {\scriptsize{$H$}};
	\node at (1.1,-.5) {\scriptsize{$\overline{H}$}};
\end{tikzpicture}
=
\sum_{\beta\in \Isom(c, a)}
\begin{tikzpicture}[baseline=-.6cm]
	\draw (.3,-1.3) arc (-180:0:.3cm) -- (.9,.3) arc (0:180:.3cm);
	\draw (-.3,0) -- (-.3,.8);
	\draw (-.3,-1) -- (-.3,-2.6);
	\draw (0,-1) -- (0,0);
	\roundNbox{unshaded}{(0,0)}{.3}{.3}{.3}{$\lambda_c(f)^*$}
	\roundNbox{unshaded}{(0,-1)}{.3}{.7}{.4}{$\lambda_c(g\circ \beta)$}
	\roundNbox{unshaded}{(-.3,-2)}{.3}{0}{0}{$\beta^*$}
	\node at (-.5,.7) {\scriptsize{$c$}};
	\node at (-.5,-1.5) {\scriptsize{$c$}};
	\node at (-.5,-2.5) {\scriptsize{$a$}};
	\node at (-.2,-.5) {\scriptsize{$H$}};
	\node at (1.1,-.5) {\scriptsize{$\overline{H}$}};
\end{tikzpicture}
=
\frac{d_H}{d_c}
\sum_{\beta \in \Isom(c, a)}
\langle f| g\circ \beta\rangle_c \beta^*
\end{align*}
by \eqref{eq:IdentitiesForA(a)}.
Plugging back in to \eqref{eq:p circ pi(g)}, we obtain
\begin{equation}
\label{eq:p circ pi(g) - 2}
p\circ \pi_a(g)
=
d_H^{-1/2}
\sum_{c\in\Irr(\cC)}
\sum_{f\in \ONB(\bfA(c))}
\sum_{\beta \in \Isom(c,a)}
\langle f| g\circ \beta\rangle_c 
\begin{tikzpicture}[baseline=.8cm]
	\draw (.3,1.1) arc (-180:0:.3cm) -- (.9,2);
	\draw (-.3,-.2) -- (-.3,1.4);
	\draw (0,1.4) --(0,2);
	\roundNbox{unshaded}{(0,1.4)}{.3}{.2}{.2}{$\lambda_c(f)$}
	\roundNbox{unshaded}{(-.3,.4)}{.3}{0}{0}{$\beta^*$}
	\node at (-.5,.9) {\scriptsize{$c$}};
	\node at (-.5,-.1) {\scriptsize{$a$}};
	\node at (-.2,1.9) {\scriptsize{$H$}};
	\node at (1.1,1.9) {\scriptsize{$\overline{H}$}};
\end{tikzpicture}.
\end{equation}
Now it is straightforward to verify that for all $f\in \bfA(c)$ and $\beta \in \cC(c,a)$, 
$\langle f|g \circ \beta\rangle_c = \langle f\circ \beta^* |g\rangle_a$.
Moreover,
$$
\set{f\circ \beta^*}{f\in \ONB(\bfA(c)),\,\, \beta \in \Isom(c,a),\,\,c\in \Irr(\cC)}
$$
is an orthonormal basis for $\bfA(a)$ under inner product \eqref{eq:L2AInnerProduct}.
Thus \eqref{eq:p circ pi(g) - 2} simplifies to
\begin{equation*}
p\circ \pi_a(g) 
= 
d_H^{-1/2} 
\sum_{h\in \ONB(\bfA(a))} 
\langle h | g\rangle_a
\begin{tikzpicture}[baseline=-.1cm]
	\draw (0,0) -- (0,.7);
	\draw (-.3,-.3) -- (-.3,-.7);
	\draw (.3,-.3) arc (-180:0:.3cm) -- (.9,.7);
	\roundNbox{unshaded}{(0,0)}{.3}{.3}{.3}{$\lambda_a(h)$}
	\node at (-.5,-.5) {\scriptsize{$a$}};
	\node at (-.2,.5) {\scriptsize{$H$}};
	\node at (1.1,0) {\scriptsize{$\overline{H}$}};
\end{tikzpicture}
=
d_H^{-1/2}
\begin{tikzpicture}[baseline=-.1cm]
	\draw (0,0) -- (0,.7);
	\draw (-.3,-.3) -- (-.3,-.7);
	\draw (.3,-.3) arc (-180:0:.3cm) -- (.9,.7);
	\roundNbox{unshaded}{(0,0)}{.3}{.3}{.3}{$\lambda_a(g)$}
	\node at (-.5,-.5) {\scriptsize{$a$}};
	\node at (-.2,.5) {\scriptsize{$H$}};
	\node at (1.1,0) {\scriptsize{$\overline{H}$}};
\end{tikzpicture}
=
\pi_a(g).
\qedhere
\end{equation*}
\end{proof}

\subsection{Morphisms}

In the last section, to each compact connected W*-algebra object $\bfA\in \Vec(\cC)$, we associated a canonical projection $p_\bfA\in \End_\cC(H_\bfA\otimes \overline{H_\bfA})$ where $H_\bfA\in \cC$ is the object representing the compact Hilbert space object $\bfL^2\bfA$, i.e., $\bfL^2\bfA = \cC(a, H_\bfA)$.
We use the notation $Q(\bfA)=\im(p_\bfA)$ for the subobject of $H_\bfA\otimes \overline{H_\bfA}$.

Suppose now we have a $*$-algebra natural transformation $\theta: \bfA \Rightarrow \bfB$ between compact connected W*-algebra objects.

\begin{nota}
Recall that for $a\in\cC$, we defined
$$
\cB_a^\bfA = \set{d_a^{1/2} \pi^\bfA_a(f)}{f \in \ONB(\bfA(a))}
$$
where $\ONB(\bfA(a))$ is an orthonormal basis for $\bfA(a)$ under the inner product \eqref{eq:L2AInnerProduct}.
By a slight abuse of notation, for $\alpha =d_a^{1/2} \pi^\bfA_a(f)\in \cB_a^\bfA$, we write 
$\theta(\alpha)$ for $d_a^{1/2} \pi^\bfB_a(\theta(f))\in \cB_a^\bfB$.
Moreover, by \eqref{eq:Isometry}, we may choose the orthonormal basis $\ONB(\bfB(a))$ so that it contains $\theta(\ONB(\bfA(a)))$ as a subset.
This means that $\theta(\cB_a^\bfA) \subset \cB_a^\bfB$ as orthonormal bases under inner product \eqref{eqn:OpenInnerProduct}.
\end{nota}

\begin{defn}
We define the following morphisms in $\cC$, which are clearly independent of the choice of orthonormal bases $\cB_a^\bfA$ for $a\in \Irr(\cC)$:
$$
Q(\theta) =
\sum_{a\in\Irr(\cC)}
\sum_{\alpha \in \cB_a^\bfA}
\begin{tikzpicture}[baseline=-.1cm]
	\draw (-.4,-.7) arc (180:0:.4cm);
	\draw (-.4,.7) arc (-180:0:.4cm);
	\draw (0,-.3) -- (0,.3);
	\filldraw[fill = white] (0,.3) circle (.05cm) node [above] {\scriptsize{$\theta(\alpha)$}};
	\filldraw[fill = white] (0,-.3) circle (.05cm) node [below] {\scriptsize{$\alpha^*$}};
	\node at (-.4,-.9) {\scriptsize{$H_\bfA$}};
	\node at (-.4,.9) {\scriptsize{$H_\bfB$}};
	\node at (.4,-.9) {\scriptsize{$\overline{H_\bfA}$}};
	\node at (.4,.9) {\scriptsize{$\overline{H_\bfB}$}};
	\node at (-.2,0) {\scriptsize{$a$}};
\end{tikzpicture}
\qquad
\qquad
\theta(p_\bfA)=
\sum_{a\in\Irr(\cC)}
\sum_{\alpha \in \cB_a^\bfA}
\begin{tikzpicture}[baseline=-.1cm]
	\draw (-.4,-.7) arc (180:0:.4cm);
	\draw (-.4,.7) arc (-180:0:.4cm);
	\draw (0,-.3) -- (0,.3);
	\filldraw[fill = white] (0,.3) circle (.05cm) node [above] {\scriptsize{$\theta(\alpha)$}};
	\filldraw[fill = white] (0,-.3) circle (.05cm) node [below] {\scriptsize{$\theta(\alpha)^*$}};
	\node at (-.4,-.9) {\scriptsize{$H_\bfB$}};
	\node at (-.4,.9) {\scriptsize{$H_\bfB$}};
	\node at (.4,-.9) {\scriptsize{$\overline{H_\bfB}$}};
	\node at (.4,.9) {\scriptsize{$\overline{H_\bfB}$}};
	\node at (-.2,0) {\scriptsize{$a$}};
\end{tikzpicture}
.
$$
It is straightforward to verify using using  \eqref{eq:Isometry} together with \eqref{eq:IdentitiesForA(a)} and \eqref{eqn:OpenInnerProduct} that $\theta(p_\bfA)$ is an orthogonal projection such that $\theta(p_\bfA)\leq p_\bfB$, and $Q(\theta)$ is a partial isometry satisfying $Q(\theta) Q(\theta)^* = \theta(p_\bfA)$ and $Q(\theta)^* Q(\theta) = p_\bfA$.
\end{defn}

We have the following key lemma.

\begin{lem}
\label{lem:RemoveTheta}
For all $a\in \cC$ and $f\in \bfA(a)$,
$$
d_{H_\bfB}^{-1/2}
\begin{tikzpicture}[baseline=-.1cm]
	\draw (0,0) -- (0,.7);
	\draw (-.3,-.3) -- (-.3,-.7);
	\draw (.3,-.3) arc (-180:0:.3cm) -- (.9,.7);
	\roundNbox{unshaded}{(0,0)}{.3}{.8}{.4}{$\lambda_a^\bfB(\theta_a(f))$}
	\node at (-.5,-.5) {\scriptsize{$a$}};
	\node at (-.3,.5) {\scriptsize{$H_\bfB$}};
	\node at (1.2,0) {\scriptsize{$\overline{H_\bfB}$}};
\end{tikzpicture}
=
d_{H_\bfB}^{-1/2}
\begin{tikzpicture}[baseline=.4cm]
	\draw (0,0) -- (0,1.7);
	\draw (-.3,-.3) -- (-.3,-.7);
	\draw (.3,-.3) arc (-180:0:.3cm) -- (.9,1.7);
	\roundNbox{unshaded}{(0,0)}{.3}{.8}{.4}{$\lambda_a^\bfB(\theta_a(f))$}
	\roundNbox{unshaded}{(.5,1)}{.3}{.4}{.3}{$p_\bfB$}
	\node at (-.5,-.5) {\scriptsize{$a$}};
	\node at (-.3,.5) {\scriptsize{$H_\bfB$}};
	\node at (1.2,0) {\scriptsize{$\overline{H_\bfB}$}};
	\node at (-.3,1.5) {\scriptsize{$H_\bfB$}};
	\node at (1.2,1.5) {\scriptsize{$\overline{H_\bfB}$}};
\end{tikzpicture}
=
d_{H_\bfB}^{-1/2}
\begin{tikzpicture}[baseline=.4cm]
	\draw (0,0) -- (0,1.7);
	\draw (-.3,-.3) -- (-.3,-.7);
	\draw (.3,-.3) arc (-180:0:.3cm) -- (.9,1.7);
	\roundNbox{unshaded}{(0,0)}{.3}{.8}{.4}{$\lambda_a^\bfB(\theta_a(f))$}
	\roundNbox{unshaded}{(.5,1)}{.3}{.4}{.3}{$\theta(p_\bfA)$}
	\node at (-.5,-.5) {\scriptsize{$a$}};
	\node at (-.3,.5) {\scriptsize{$H_\bfB$}};
	\node at (1.2,0) {\scriptsize{$\overline{H_\bfB}$}};
	\node at (-.3,1.5) {\scriptsize{$H_\bfB$}};
	\node at (1.2,1.5) {\scriptsize{$\overline{H_\bfB}$}};
\end{tikzpicture}
=
d_{H_\bfA}^{-1/2}
\begin{tikzpicture}[baseline=.4cm]
	\draw (0,0) -- (0,1.7);
	\draw (-.3,-.3) -- (-.3,-.7);
	\draw (.3,-.3) arc (-180:0:.3cm) -- (.9,1.7);
	\roundNbox{unshaded}{(0,0)}{.3}{.4}{.4}{$\lambda_a^\bfA(f)$}
	\roundNbox{unshaded}{(.5,1)}{.3}{.4}{.3}{$Q(\theta)$}
	\node at (-.5,-.5) {\scriptsize{$a$}};
	\node at (-.3,.5) {\scriptsize{$H_\bfA$}};
	\node at (1.2,0) {\scriptsize{$\overline{H_\bfA}$}};
	\node at (-.3,1.5) {\scriptsize{$H_\bfB$}};
	\node at (1.2,1.5) {\scriptsize{$\overline{H_\bfB}$}};
\end{tikzpicture}.
$$
In particular, $\pi_a^\bfB(\theta(f)) = Q(\theta) \circ \pi^\bfA_a(f)$.
\end{lem}
\begin{proof}
The first equality follows from Lemma \ref{lem:RemovePFromAMorphisms}.
The second equality follows from the fact that $\pi_a^\bfB(\theta_a(f))$ is orthogonal to $\cB^\bfB_a \setminus \theta(\cB_a^\bfA)$.
We omit the proof of the final equality, which is similar to the proof of Lemma \ref{lem:RemovePFromAMorphisms} using \eqref{eq:Isometry} together with \eqref{eq:IdentitiesForA(a)} and \eqref{eqn:OpenInnerProduct}.
\end{proof}

\begin{cor}
The morphism $Q(\theta) \in \cC(Q(\bfA), Q(\bfB))$ is an involutive algebra morphism.
\end{cor}
\begin{proof}
It is easy to see that $Q(\theta)$ is unital, since $\theta_{1_\cC}(i_\bfA) = i_\bfB$.
Thus by Lemma \ref{lem:RemoveTheta}, 
$$
d_{H_\bfA}^{-1/2}
Q(\theta)\circ \coev_{H_\bfA} 
= 
Q(\theta) \circ \pi_{1_\cC}^\bfA(i_\bfA) 
= 
\pi_{1_\cC}^\bfB(\theta(i_\bfA)) 
= 
\pi_{1_\cC}^\bfB(i_\bfB) 
= 
d_{H_\bfB}^{-1/2}
\coev_{H_\bfB}.
$$
To show $Q(\theta)$ is multiplicative, we first use multiplicativity of $\theta : \bfA \Rightarrow \bfB$ and $\pi^\bfB: \bfB\Rightarrow \bfL^2\bfB\otimes \overline{\bfL^2\bfB}$ 
and the relationship between the inner products \eqref{eq:Isometry}, \eqref{eq:IdentitiesForA(a)},  and \eqref{eqn:OpenInnerProduct}
to prove an identity similar to \eqref{eq:ReplaceABC}.
For all $a,b,c\in \Irr(\cC)$, $\alpha \in \cB_a$, $\beta \in \cB_b$, and $\gamma\in \cB_c$, 
expanding using the inner product on $\Isom(c, a\otimes b)$ yields
\begin{equation}
\label{eq:ReplaceThetaABC}
\begin{tikzpicture}[baseline=-.1cm]
	\draw (0,.6) -- (0,1.2);
	\draw (-.6,-.4) -- (-.6,-1);
	\draw (.6,-.4) -- (.6,-1);
	\draw (-.95,-.1) .. controls ++(90:.1cm) and ++(270:.2cm) .. (-.35,.3);
	\draw (.95,-.1) .. controls ++(90:.1cm) and ++(270:.2cm) .. (.35,.3);
	\draw (-.25,-.1) .. controls ++(90:.2cm) and ++(90:.2cm) .. (.25,-.1);
	\roundNbox{unshaded}{(-.6,-.4)}{.3}{.2}{.2}{$\theta(\alpha)$};
	\roundNbox{unshaded}{(.6,-.4)}{.3}{.2}{.2}{$\theta(\beta)$};
	\roundNbox{unshaded}{(0,.6)}{.3}{.2}{.2}{$\theta(\gamma)^*$};
	\node at (-1,.2) {\scriptsize{$H_\bfB$}};
	\node at (1,.2) {\scriptsize{$\overline{H_\bfB}$}};
	\node at (-.2,1.1) {\scriptsize{$c$}};
	\node at (-.8,-.9) {\scriptsize{$a$}};
	\node at (.8,-.9) {\scriptsize{$b$}};
\end{tikzpicture}
=
\frac{d_{H_\bfA}^{1/2}}{d_c d_{H_\bfB}^{1/2}}
\sum_{\delta \in \Isom(c, a\otimes b)}
\Delta\left(
\begin{array}{ccc} a & b & c \\ \alpha & \beta & \gamma \end{array}
\middle| 
\begin{array}{c} \delta \end{array}
\right)
\begin{tikzpicture}[baseline=-.1cm]
	\draw (-.3,-.3) arc (180:0:.3cm);
	\draw (0,0) -- (0,.3);
	\filldraw[fill = white] (0,0) circle (.05cm) node [below] {\scriptsize{$\delta^*$}};
	\node at (-.3,-.5) {\scriptsize{$a$}};
	\node at (.3,-.5) {\scriptsize{$b$}};
	\node at (0,.5) {\scriptsize{$c$}};
\end{tikzpicture}
\underset{\eqref{eq:ReplaceABC}}{=}
\frac{d_{H_\bfA}^{1/2}}{d_{H_\bfB}^{1/2}}
\begin{tikzpicture}[baseline=-.1cm]
	\draw (0,.6) -- (0,1.2);
	\draw (-.4,-.4) -- (-.4,-1);
	\draw (.4,-.4) -- (.4,-1);
	\draw (-.55,-.1) .. controls ++(90:.1cm) and ++(270:.2cm) .. (-.15,.3);
	\draw (.55,-.1) .. controls ++(90:.1cm) and ++(270:.2cm) .. (.15,.3);
	\draw (-.25,-.1) .. controls ++(90:.2cm) and ++(90:.2cm) .. (.25,-.1);
	\roundNbox{unshaded}{(-.4,-.4)}{.3}{0}{0}{$\alpha$};
	\roundNbox{unshaded}{(.4,-.4)}{.3}{0}{0}{$\beta$};
	\roundNbox{unshaded}{(0,.6)}{.3}{0}{0}{$\gamma^*$};
	\node at (-.7,.2) {\scriptsize{$H_\bfA$}};
	\node at (.7,.2) {\scriptsize{$\overline{H_\bfA}$}};
	\node at (-.2,1.1) {\scriptsize{$c$}};
	\node at (-.6,-.9) {\scriptsize{$a$}};
	\node at (.6,-.9) {\scriptsize{$b$}};
\end{tikzpicture}
.
\end{equation}
Now, similar to the proof of Lemma \ref{lem:RemoveTopP}, we have
\begin{align*}
m_{Q(\bfB)} \circ &(Q(\theta) \otimes Q(\theta))
\\&:=
d_{H_\bfB}^{1/2}
\begin{tikzpicture}[baseline=-.1cm]
	\draw (-.95,-1) -- (-.95,-.1) .. controls ++(90:.1cm) and ++(270:.2cm) .. (-.35,.3) -- (-.35,1.2);
	\draw (.95,-1) -- (.95,-.1) .. controls ++(90:.1cm) and ++(270:.2cm) .. (.35,.3) -- (.35,1.2);
	\draw (-.25,-1) -- (-.25,-.1) .. controls ++(90:.2cm) and ++(90:.2cm) .. (.25,-.1) -- (.25,-1);
	\roundNbox{unshaded}{(-.6,-.4)}{.3}{.2}{.2}{$Q(\theta)$};
	\roundNbox{unshaded}{(.6,-.4)}{.3}{.2}{.2}{$Q(\theta)$};
	\roundNbox{unshaded}{(0,.6)}{.3}{.2}{.2}{$p_\bfB$};
\end{tikzpicture}
=
d_{H_\bfB}^{1/2}
\begin{tikzpicture}[baseline=-.1cm]
	\draw (-.95,-1) -- (-.95,-.1) .. controls ++(90:.1cm) and ++(270:.2cm) .. (-.35,.3) -- (-.35,1.2);
	\draw (.95,-1) -- (.95,-.1) .. controls ++(90:.1cm) and ++(270:.2cm) .. (.35,.3) -- (.35,1.2);
	\draw (-.25,-1) -- (-.25,-.1) .. controls ++(90:.2cm) and ++(90:.2cm) .. (.25,-.1) -- (.25,-1);
	\roundNbox{unshaded}{(-.6,-.4)}{.3}{.2}{.2}{$Q(\theta)$};
	\roundNbox{unshaded}{(.6,-.4)}{.3}{.2}{.2}{$Q(\theta)$};
	\roundNbox{unshaded}{(0,.6)}{.3}{.3}{.3}{$\theta(p_\bfA)$};
\end{tikzpicture}
&& 
\text{(Lem.~\ref{lem:RemoveTopP})}
\\&=
d_{H_{\bfA}}^{1/2}
\sum_{a,b,c\in \Irr(\cC)} 
d_c^{-1}
\sum_{\substack{
\alpha\in \cB_a
\\
\beta\in\cB_b
\\
\gamma\in\cB_c
}}
\sum_{\delta\in \Isom(c, a\otimes b)}
\Delta\left(
\begin{array}{ccc} a & b & c \\ \alpha & \beta & \gamma \end{array}
\middle| 
\begin{array}{c} \delta \end{array}
\right)
\begin{tikzpicture}[baseline=.3cm]
	\draw (-.9,-.6) arc (180:0:.3cm);
	\draw (.3,-.6) arc (180:0:.3cm);
	\draw (-.6,-.3) arc (180:0:.6cm);
	\draw (0,.3) -- (0,.9);
	\draw (-.4,1.3) arc (-180:0:.4cm);
	\filldraw[fill = white] (-.6,-.3) circle (.05cm) node [below] {\scriptsize{$\alpha^*$}};
	\filldraw[fill = white] (.6,-.3) circle (.05cm) node [below] {\scriptsize{$\beta^*$}};
	\filldraw[fill = white] (0,.3) circle (.05cm) node [below] {\scriptsize{$\delta^*$}};
	\filldraw[fill = white] (0,.9) circle (.05cm) node [above] {\scriptsize{$\theta(\gamma)$}};
	\node at (-.9,-.8) {\scriptsize{$H_\bfA$}};
	\node at (-.3,-.8) {\scriptsize{$\overline{H_\bfA}$}};
	\node at (.9,-.8) {\scriptsize{$\overline{H_\bfA}$}};
	\node at (.3,-.8) {\scriptsize{$H_\bfA$}};
	\node at (.3,1.5) {\scriptsize{$\overline{H_\bfB}$}};
	\node at (-.3,1.5) {\scriptsize{$H_\bfB$}};
	\node at (-.8,-.1) {\scriptsize{$a$}};
	\node at (.8,-.1) {\scriptsize{$b$}};
	\node at (.2,.6) {\scriptsize{$c$}};
\end{tikzpicture}
&&
\text{(by \eqref{eq:ReplaceThetaABC})}
\\&=
d_{H_\bfA}^{1/2}
\begin{tikzpicture}[baseline=-.1cm]
	\draw (-.55,-1) -- (-.55,-.1) .. controls ++(90:.1cm) and ++(270:.2cm) .. (-.15,.3) -- (-.15,1.2);
	\draw (.55,-1) -- (.55,-.1) .. controls ++(90:.1cm) and ++(270:.2cm) .. (.15,.3) -- (.15,1.2);
	\draw (-.25,-1) -- (-.25,-.1) .. controls ++(90:.2cm) and ++(90:.2cm) .. (.25,-.1) -- (.25,-1);
	\roundNbox{unshaded}{(-.4,-.4)}{.3}{0}{0}{$p_\bfA$};
	\roundNbox{unshaded}{(.4,-.4)}{.3}{0}{0}{$p_\bfA$};
	\roundNbox{unshaded}{(0,.6)}{.3}{.2}{.2}{$Q(\theta)$};
\end{tikzpicture}
=:
Q(\theta) \circ m_{Q(\bfA)}
&&
\text{(by \eqref{eq:ReplaceABC}).}
\end{align*}

Finally, $\sigma_{Q(\bfA)} = p_\bfA$ and $\sigma_{Q(\bfB)} = p_\bfB$, so invoutivity of $Q(\theta)$ is equivalent to $Q(\theta) = \overline{Q(\theta)}$, which is easily seen to hold by independence of the choice of $\cB_a^\bfA$ as in the proof of Lemma \ref{lem:SymmetricallySelfDualProjector}.
(This suppresses the isomorphism between real objects of the form $c\otimes \overline{c}\in \cC$ and their conjugates as in Example \ref{ex:StarAlgebraFromInnerHom}.)
\end{proof}

\begin{prop}
$Q$ is a functor.
\end{prop}
\begin{proof}
It remains to show that $Q$ preserves identities and composites.
First, if $\theta = \id_\bfA$, then $Q(\theta) = p_\bfA = \id_{Q(\bfA)}$.
Second, if $\theta_1: \bfA \Rightarrow \bfB$ and $\theta_2 : \bfB \Rightarrow \bfC$, then for each $a\in \cC$, we choose $\ONB(\cC(a))$ so as to include $\theta_2(\ONB(\bfB(a)))$ just as we chose $\ONB(\bfB(a))$ to include $\theta_1(\ONB(\bfA(a))$.
Now for all $\alpha \in \cB^\bfA_a$, $\theta_1(\alpha)$ is orthogonal to $\cB^\bfB_a\setminus \theta_1(\cB_a^\bfA)$ under inner product \eqref{eqn:OpenInnerProduct}, and thus
\begin{align*}
Q(\theta_2)Q(\theta_1) 
&=
\sum_{a,b\in\Irr(\cC)}
\sum_{\substack{
\alpha \in \cB_a^\bfA
\\
\beta \in \cB_b^\bfB
}}
\begin{tikzpicture}[baseline=-.1cm]
	\draw (-.4,1.8) -- (-.4,2);
	\draw (.4,1.8) -- (.4,2);
	\draw (-.4,-.2) -- (-.4,.2);
	\draw (.4,-.2) -- (.4,.2);
	\draw (-.4,-1.8) -- (-.4,-2);
	\draw (.4,-1.8) -- (.4,-2);
	\draw (0,.8) -- (0,1.2);
	\draw (0,-.8) -- (0,-1.2);
	\roundNbox{unshaded}{(0,1.5)}{.3}{.4}{.4}{$\theta_2(\beta)$}
	\roundNbox{unshaded}{(0,.5)}{.3}{.4}{.4}{$\beta^*$}
	\roundNbox{unshaded}{(0,-.5)}{.3}{.4}{.4}{$\theta_1(\alpha)$}
	\roundNbox{unshaded}{(0,-1.5)}{.3}{.4}{.4}{$\alpha^*$}
	\node at (-.4,-2.2) {\scriptsize{$H_\bfA$}};
	\node at (-.7,0) {\scriptsize{$H_\bfB$}};
	\node at (-.4,2.2) {\scriptsize{$H_\bfC$}};
	\node at (.4,-2.2) {\scriptsize{$\overline{H_\bfA}$}};
	\node at (.7,0) {\scriptsize{$\overline{H_\bfB}$}};
	\node at (.4,2.2) {\scriptsize{$\overline{H_\bfC}$}};
	\node at (-.2,1) {\scriptsize{$b$}};
	\node at (-.2,-1) {\scriptsize{$a$}};
\end{tikzpicture}
=
\sum_{a\in\Irr(\cC)}
\sum_{\alpha \in \cB_a^\bfA}
\begin{tikzpicture}[baseline=-.1cm]
	\draw (-.4,.8) -- (-.4,1);
	\draw (.4,.8) -- (.4,1);
	\draw (-.4,-.8) -- (-.4,-1);
	\draw (.4,-.8) -- (.4,-1);
	\draw (0,-.2) -- (0,.2);
	\roundNbox{unshaded}{(0,.5)}{.3}{.6}{.6}{$\theta_2(\theta_1(\alpha))$}
	\roundNbox{unshaded}{(0,-.5)}{.3}{.4}{.4}{$\alpha^*$}
	\node at (-.4,-1.2) {\scriptsize{$H_\bfA$}};
	\node at (-.4,1.2) {\scriptsize{$H_\bfC$}};
	\node at (.4,-1.2) {\scriptsize{$\overline{H_\bfA}$}};
	\node at (.4,1.2) {\scriptsize{$\overline{H_\bfC}$}};
	\node at (-.2,0) {\scriptsize{$a$}};
\end{tikzpicture}
=
Q(\theta_2\circ \theta_1).
\qedhere
\end{align*}
\end{proof}

\section{Equivalence of Q-systems and W*-algebra objects}

We now prove the the functors $\bfW\text{*}$ and $Q$ constructed in Sections \ref{sec:QtoW} and \ref{sec:WtoQ} respectively witness an equivalence of categories.

\subsection{From W*-algebras to Q-systems and back}
\label{sec:WtoQtoW}

We begin by building a natural isomorphism $\eta : \id \Rightarrow \WStar\circ Q$.

Suppose we have a compact connected W*-algebra object $(\bfA, \mu^\bfA, i_\bfA, j^\bfA) \in \Vec(\cC)$ together with its canonical GNS faithful $*$-algebra natural transformation $\lambda: \bfA \Rightarrow \bfB(\bfL^2\bfA)$.
Recall that $Q(\bfA) = \im(p_\bfA) \subset H_\bfA\otimes \overline{H_\bfA}\in \cC$ is an irreducible Q-system by Proposition \ref{prop:W*ToFrobenius}.
Consider the corresponding compact connected W*-algebra object $\WStar Q(\bfA) \in \Vec(\cC)$ as in Proposition \ref{prop:FrobeniusToW*}, and notice that for all $a\in \cC$, $\WStar Q(\bfA)(a) = \cC(a, Q(\bfA))$ with multiplication and unit induced from $Q(\bfA)$, and $*$-structure given by conjugation, which is induced from the adjoint on the cyclic $\cC$-module W*-category $(\cC, B)$.

\begin{thm}
\label{thm:W*toQandBack}
The GNS representation $\lambda: \bfA\Rightarrow \bfB(\bfL^2\bfA)$ induces a canonical $*$-algebra natural isomorphism $\eta^\bfA:\bfA \Rightarrow \bfW$\emph{*}$Q(\bfA)$.
\end{thm}
\begin{proof}
Using the shorthand $\bfH_\bfA = \bfL^2\bfA$, let $\pi^\bfA: \bfA \Rightarrow \bfH_\bfA\otimes \overline{\bfH_\bfA}$ be the induced faithful $*$-algebra natural transformation as in Definition \ref{defn:InducedRepresentation}.
Note that $\pi_a^\bfA(\bfA(a))$ is exactly equal to
$$
\WStar Q(a) = \cC(a, Q(\bfA)) =\cC(a, \im(p_\bfA))\subset \cC(a, H_\bfA\otimes \overline{H_\bfA}).
$$
We define $\eta^\bfA : \bfA \Rightarrow \WStar Q(\bfA)$ by $\eta^\bfA_a(f) = p_\bfA \circ \pi^\bfA_a(f) \in \cC(a, \im(p_\bfA))$, i.e., $\eta^\bfA$ is just $\pi^\bfA$ considered as a morphism $\bfA \Rightarrow \WStar Q(\bfA)$ rather than $\bfA \Rightarrow \bfH_\bfA \otimes \overline{\bfH_\bfA}$.

To show that $\eta^\bfA: \bfA \Rightarrow \WStar Q(\bfA)$ is a $*$-algebra natural isomorphism, it remains to show that $\eta^\bfA$ intertwines the unit,  multiplication, and $*$-structures of $\bfA$ and $\WStar Q(\bfA)$.
As in Lemma \ref{lem:RemovePFromAMorphisms}, 
$$
\pi_{1_\cC}^\bfA(i_\bfA) = 
d_{H_\bfA}^{-1/2}
\begin{tikzpicture}[baseline=-.1cm]
	\draw (0,0) -- (0,.7);
	\draw[dotted, thick] (-.3,-.3) -- (-.3,-.7);
	\draw (.3,-.3) arc (-180:0:.3cm) -- (.9,.7);
	\roundNbox{unshaded}{(0,0)}{.3}{.4}{.4}{$\lambda^\bfA_a(i_\bfA)$}
	\node at (-.5,-.5) {\scriptsize{$1$}};
	\node at (-.3,.5) {\scriptsize{$H_\bfA$}};
	\node at (1.2,0) {\scriptsize{$\overline{H_\bfA}$}};
\end{tikzpicture}
=
d_{H_\bfA}^{-1/2}
\coev_{H_\bfA}
=
i_{\WStar Q(\bfA)}.
$$
If $f\in \bfA(a)$ and $g\in \bfA(b)$, then we see that
\begin{align*}
\mu^{\WStar Q(\bfA)}_{a,b}(\pi^\bfA_a(f)\otimes \pi^\bfA_b(g))
&=
\frac{d_{H_\bfA}^{1/2}}{d_{H_\bfA}}
\begin{tikzpicture}[baseline=-.1cm]
	\draw (0,0) -- (0,.7);
	\draw (-.3,-.3) -- (-.3,-.7);
	\draw (1.7,-.3) -- (1.7,-.7);
	\draw (.3,-.3) arc (-180:0:.3cm) -- (.9,.3) .. controls ++(90:.5cm) and ++(90:.5cm) .. (2,.3);
	\draw (2.3,-.3) arc (-180:0:.3cm) -- (2.9,.7);
	\roundNbox{unshaded}{(0,0)}{.3}{.4}{.4}{$\lambda^\bfA_a(f)$}
	\roundNbox{unshaded}{(2,0)}{.3}{.4}{.4}{$\lambda^\bfA_b(g)$}
	\node at (-.5,-.5) {\scriptsize{$a$}};
	\node at (1.5,-.5) {\scriptsize{$b$}};
	\node at (-.3,.5) {\scriptsize{$H_\bfA$}};
	\node at (.6,.5) {\scriptsize{$\overline{H_\bfA}$}};
	\node at (3.2,0) {\scriptsize{$\overline{H_\bfA}$}};
\end{tikzpicture}
=
d_{H_\bfA}^{-1/2}
\begin{tikzpicture}[baseline=.4cm]
	\draw (0,0) -- (0,1.7);
	\draw (-.3,-.3) -- (-.3,-.7);
	\draw (-.8,.7) -- (-.8,-.7);
	\draw (.3,-.3) arc (-180:0:.3cm) -- (.9,1.7);
	\roundNbox{unshaded}{(0,0)}{.3}{.3}{.3}{$\lambda^\bfA_b(g)$}
	\roundNbox{unshaded}{(0,1)}{.3}{.8}{.3}{$\lambda^\bfA_a(f)$}
	\node at (-1,-.5) {\scriptsize{$a$}};
	\node at (-.5,-.5) {\scriptsize{$b$}};
	\node at (-.3,.5) {\scriptsize{$H_\bfA$}};
	\node at (-.3,1.5) {\scriptsize{$H_\bfA$}};
	\node at (1.2,.5) {\scriptsize{$\overline{H_\bfA}$}};
\end{tikzpicture}
\\&=
d_{H_\bfA}^{-1/2}
\begin{tikzpicture}[baseline=-.1cm]
	\draw (0,0) -- (0,.7);
	\draw (-1.2,-.3) -- (-1.2,-.7);
	\draw (-.8,-.3) -- (-.8,-.7);
	\draw (1.2,-.3) arc (-180:0:.3cm) -- (1.8,.7);
	\roundNbox{unshaded}{(0,0)}{.3}{1.4}{1.3}{$\lambda^\bfA_{a\otimes b}(\mu^\bfA_{a\otimes b}(f\otimes g))$}
	\node at (-1,-.5) {\scriptsize{$a$}};
	\node at (-.6,-.5) {\scriptsize{$b$}};
	\node at (-.3,.5) {\scriptsize{$H_\bfA$}};
	\node at (2.1,0) {\scriptsize{$\overline{H_\bfA}$}};
\end{tikzpicture}
=
\pi_{a\otimes b}^\bfA(\mu^\bfA_{a,b}(f\otimes g)).
\end{align*}
Finally, if $f\in \bfA(a)$, using \eqref{eq:BarStar}, together with the fact that $\lambda$ is preserves the $*$-structure, and the correspondence between the $*$-structure of $\bfA$ and the dagger structure of $\cM_\bfA$ from \cite[Thm.~3.20]{1611.04620}, we have that
\begin{align*}
j_a^{\WStar Q(A)}(\pi_a^\bfA(f)) 
&=
d_{H_\bfA}^{-1/2}
\begin{tikzpicture}[baseline=-.1cm, xscale =-1]
	\draw (0,0) -- (0,.7);
	\draw (-.3,-.3) -- (-.3,-.7);
	\draw (.3,-.3) arc (-180:0:.3cm) -- (.9,.7);
	\roundNbox{unshaded}{(0,0)}{.3}{.4}{.4}{$\overline{\lambda_a^\bfA(f)}$}
	\node at (-.5,-.5) {\scriptsize{$\overline{a}$}};
	\node at (-.3,.6) {\scriptsize{$\overline{H_\bfA}$}};
	\node at (1.2,.5) {\scriptsize{$H_\bfA$}};
\end{tikzpicture}
=
d_{H_\bfA}^{-1/2}
\begin{tikzpicture}[baseline=-.1cm, xscale =-1]
	\draw (0,0) -- (0,.7);
	\draw (-.6,-.3) -- (-.6,-.7);
	\draw (.6,-.3) arc (-180:0:.3cm) -- (1.2,.7);
	\roundNbox{unshaded}{(0,0)}{.3}{.6}{.6}{$(\lambda^\bfA_a(f)^*)^\vee$}
	\node at (-.8,-.5) {\scriptsize{$\overline{a}$}};
	\node at (-.3,.6) {\scriptsize{$\overline{H_\bfA}$}};
	\node at (1.5,.5) {\scriptsize{$H_\bfA$}};
\end{tikzpicture}
\\&=
d_{H_\bfA}^{-1/2}
\begin{tikzpicture}[baseline=-.1cm]
	\draw (0,-.3) .. controls ++(270:.5cm) and ++(270:.5cm) .. (.9,-.3) -- (.9,.7);
	\draw (.3,.3) -- (.3,.7);
	\draw (-.3,.3) arc (0:180:.3cm) -- (-.9,-.7);
	\roundNbox{unshaded}{(0,0)}{.3}{.4}{.4}{$\lambda^\bfA_a(f)^*$}
	\node at (-1.1,-.5) {\scriptsize{$\overline{a}$}};
	\node at (1.2,.5) {\scriptsize{$\overline{H_\bfA}$}};
	\node at (.6,.5) {\scriptsize{$H_\bfA$}};
\end{tikzpicture}
=
d_{H_\bfA}^{-1/2}
\begin{tikzpicture}[baseline=-.1cm]
	\draw (0,0) -- (0,.7);
	\draw (-.5,-.3) -- (-.5,-.7);
	\draw (.5,-.3) arc (-180:0:.3cm) -- (1.1,.7);
	\roundNbox{unshaded}{(0,0)}{.3}{.7}{.6}{$\lambda^\bfA_{\overline{a}}(j_a^\bfA(f))$}
	\node at (-.7,-.5) {\scriptsize{$\overline{a}$}};
	\node at (-.3,.6) {\scriptsize{$\overline{H_\bfA}$}};
	\node at (1.4,.5) {\scriptsize{$H_\bfA$}};
\end{tikzpicture}
=
\pi^\bfA_{\overline{a}}(j^\bfA_a(f)).
\end{align*}
This completes the proof.
\end{proof}

\begin{prop}
The $*$-algebra isomorphism $\eta^\bfA : \bfA \Rightarrow \bfW$\emph{*}$Q(\bfA)$ is natural in $\bfA$.
\end{prop}
\begin{proof}
Suppose $\theta: \bfA \Rightarrow \bfB$ is a natural $*$-algebra transformation.
We must prove that $\WStar Q(\theta) \circ \eta^\bfA = \eta^\bfB \circ \theta$.
We calculate that for $a\in \cC$ and $f\in \bfA(a)$, we have
$$
(\WStar Q(\theta)_a \circ \eta^\bfA_a)(f)
=
d_{H_\bfA}^{-1/2}
\begin{tikzpicture}[baseline=.4cm]
	\draw (0,0) -- (0,1.7);
	\draw (-.3,-.3) -- (-.3,-.7);
	\draw (.3,-.3) arc (-180:0:.3cm) -- (.9,1.7);
	\roundNbox{unshaded}{(0,0)}{.3}{.4}{.4}{$\lambda_a^\bfA(f)$}
	\roundNbox{unshaded}{(.5,1)}{.3}{.4}{.3}{$Q(\theta)$}
	\node at (-.5,-.5) {\scriptsize{$a$}};
	\node at (-.3,.5) {\scriptsize{$H_\bfA$}};
	\node at (1.2,0) {\scriptsize{$\overline{H_\bfA}$}};
	\node at (-.3,1.5) {\scriptsize{$H_\bfB$}};
	\node at (1.2,1.5) {\scriptsize{$\overline{H_\bfB}$}};
\end{tikzpicture}
\qquad\text{and}\qquad
(\eta^\bfB_a \circ \theta_a) (f)
=
d_{H_\bfB}^{-1/2}
\begin{tikzpicture}[baseline=-.1cm]
	\draw (0,0) -- (0,.7);
	\draw (-.3,-.3) -- (-.3,-.7);
	\draw (.3,-.3) arc (-180:0:.3cm) -- (.9,.7);
	\roundNbox{unshaded}{(0,0)}{.3}{.8}{.4}{$\lambda_a^\bfB(\theta_a(f))$}
	\node at (-.5,-.5) {\scriptsize{$a$}};
	\node at (-.3,.5) {\scriptsize{$H_\bfB$}};
	\node at (1.2,0) {\scriptsize{$\overline{H_\bfB}$}};
\end{tikzpicture}
.
$$
The equality of the above morphisms is exactly the identity $\pi^\bfB_a(\theta(f)) = Q(\theta)\circ \pi^\bfA_a(f)$ which was established in Lemma \ref{lem:RemoveTheta}.
\end{proof}

\subsection{From Q-systems to W*-algebras and back}
\label{sec:QtoWtoQ}

We now build a natural isomorphism $\zeta: \id \Rightarrow Q\WStar$.

Starting with a normalized irreducible Q-system $(A,m_A,i_A) \in \cC$, define the compact connected W*-algebra object $\bfA \in \Vec(\cC)$ by $\bfA(a) = \cC^{\natural}(a, A)$ as in Proposition \ref{prop:FrobeniusToW*}.
By Lemma \ref{lem:H=A}, the compact Hilbert space object $\bfL^2\bfA$ is given by $\bfL^2\bfA(a) = \cC(a, A)$, and by Lemma  \ref{lem:RemovePFromAMorphisms}, the canonical GNS faithful $*$-algebra natural transformation $\lambda: \bfA \Rightarrow \bfB(\bfL^2\bfA)$ is given by \eqref{eq:ExplicitGNS}.
By Proposition \ref{prop:W*ToFrobenius},  $Q\WStar(A) = \im(p) \subset H\otimes \overline{H}\in \cC$ is a normalized irreducible Q-system.

\begin{thm}
Under the GNS representation $\lambda: \bfA \Rightarrow \bfB(\bfL^2\bfA)$, the identity map $\id_A \in \cC(A,A)$ induces a canonical unitary involutive algebra isomorphism $\zeta_A\in \cC(A , Q\WStar(A))$.
\end{thm}
\begin{proof}
Recall that $\cC(A, Q\WStar(A)) = \cC(A, \im(p)) \subset \cC(A, A\otimes \overline{A})$, and that $\cC(A, \im(p))$ is exactly the image of $\pi_A : \bfA(A) \to (\bfL^2\bfA\otimes \overline{\bfL^2\bfA})(A) = \cC(A, A\otimes \overline{A})$ by Lemma \ref{lem:RemovePFromAMorphisms}.
By Remark \ref{rem:ExplicitGNS},
$$
\lambda_A(\id_A) 
= 
\begin{tikzpicture}[baseline=-.1cm]
	\draw (.3,-.6) -- (.3,.6);
	\filldraw (.3,0) circle (.05cm);
	\draw (-.1,-.6) .. controls ++(90:.3cm) and ++(-135:.3cm) .. (.3,0);
	\node at (.5,-.5) {\scriptsize{$A$}};
	\node at (-.3,-.5) {\scriptsize{$A$}};
	\node at (.5,.5) {\scriptsize{$A$}};
\end{tikzpicture}
\in \cC(A\otimes A, A)
\qquad
\Longrightarrow
\qquad
\pi_A(\id_A) 
=
d_A^{-1/2}
\begin{tikzpicture}[baseline=-.1cm]
	\draw (.8,.6) -- (.8,-.3) arc (0:-180:.25cm) (.3,-.3) -- (.3,.6);
	\filldraw (.3,0) circle (.05cm);
	\draw (0,-.6) .. controls ++(90:.3cm) and ++(-135:.3cm) .. (.3,0);
	\node at (.5,.5) {\scriptsize{$A$}};
	\node at (-.2,-.5) {\scriptsize{$A$}};
	\node at (1.1,.5) {\scriptsize{$\overline{A}$}};
\end{tikzpicture}
=:\zeta_A
$$
Denoting $\zeta_A:= \pi_A(\id_A)$, by Lemma \ref{lem:RemovePFromAMorphisms}, $\zeta_A=p\circ \zeta_A\in \cC(A, \im(p)) = \cC(A, Q\WStar(A))$.

We claim $\zeta_A$ is an involutive unitary isomorphism of Q-systems.
First, since $(A,m_A,i_A)$ is a normalized Q-system, $m_A\circ m_A^* = d_A \id_A$.
Thus $\zeta_A^*\circ \zeta_A = \id_A$ by construction.
Now $\zeta_A = p\circ \zeta_A$ by Lemma \ref{lem:RemovePFromAMorphisms}, and it is easy to see that $\zeta_A\circ \zeta_A^* = \zeta_A\circ \zeta_A^* = p = \id_{Q\WStar(A)}$.
Indeed, $q:= \zeta_A\circ \zeta_A^*$ is a subprojection of $p$ such that $\dim(\im(q)) = \dim(A) = \dim(\im(p))$, which means $q=p$.
Thus $\zeta_A$ is a unitary isomorphism.

It remains to show that $\zeta_A$ intertwines the unit and multiplications of $A$ and $B$.
First, by Lemma \ref{lem:RemovePFromAMorphisms}, 
$$
\zeta_A\circ i_A 
=
d_A^{-1/2}
\begin{tikzpicture}[baseline=-.1cm]
	\draw (.8,.6) -- (.8,-.3) arc (0:-180:.25cm) (.3,-.3) -- (.3,.6);
	\filldraw (.3,0) circle (.05cm);
	\filldraw (0,-.6) circle (.05cm) node [left] {\scriptsize{$i_A$}};
	\draw (0,-.6) .. controls ++(90:.3cm) and ++(-135:.3cm) .. (.3,0);
	\node at (.5,.5) {\scriptsize{$A$}};
	\node at (1.1,.5) {\scriptsize{$\overline{A}$}};
\end{tikzpicture}
=
d_A^{-1/2} \coev_A
= 
d_A^{-1/2} (p\circ \coev_A)
=
i_{Q\WStar(A)}.
$$
Second, using Corollary \ref{cor:RemoveAnyP}, $p\circ \zeta_A = \zeta_A$ by Lemma \ref{lem:RemovePFromAMorphisms}, and the Frobenius relation,
$$
m_{Q\WStar(A)} \circ(\zeta_A\otimes \zeta_A)
=
\frac{d_A^{1/2}}{d_A}
\begin{tikzpicture}[baseline=-.4cm]
	\draw (-.55,-.1) .. controls ++(90:.1cm) and ++(270:.2cm) .. (-.15,.3) -- (-.15,1.2);
	\draw (.55,-.1) .. controls ++(90:.1cm) and ++(270:.2cm) .. (.15,.3) -- (.15,1.2);
	\draw (-.25,-.1) .. controls ++(90:.2cm) and ++(90:.2cm) .. (.25,-.1);
	\draw (-.25,-.7) -- (-.25,-1.1) arc (0:-180:.15cm) -- (-.55,-.7);
	\draw (.25,-.7) -- (.25,-1.1) arc (-180:0:.15cm) -- (.55,-.7);
	\filldraw (-.55,-.9) circle (.05cm);
	\filldraw (.25,-.9) circle (.05cm);
	\draw (-.8,-1.5) .. controls ++(90:.3cm) and ++(-135:.2cm) .. (-.55,-.9);
	\draw (0,-1.5) .. controls ++(90:.3cm) and ++(-135:.2cm) .. (.25,-.9);
	\roundNbox{unshaded}{(-.4,-.4)}{.3}{0}{0}{$p$};
	\roundNbox{unshaded}{(.4,-.4)}{.3}{0}{0}{$p$};
	\roundNbox{unshaded}{(0,.6)}{.3}{0}{0}{$p$};
\end{tikzpicture}
=
d_A^{-1/2}
\begin{tikzpicture}[baseline=-.4cm]
	\draw (-.55,-.1) -- (-.55,.4);
	\draw (.55,-.1) -- (.55,.4);
	\draw (-.25,-.1) .. controls ++(90:.2cm) and ++(90:.2cm) .. (.25,-.1);
	\draw (-.25,-.7) -- (-.25,-1.1) arc (0:-180:.15cm) -- (-.55,-.7);
	\draw (.25,-.7) -- (.25,-1.1) arc (-180:0:.15cm) -- (.55,-.7);
	\filldraw (-.55,-.9) circle (.05cm);
	\filldraw (.25,-.9) circle (.05cm);
	\draw (-.8,-1.5) .. controls ++(90:.3cm) and ++(-135:.2cm) .. (-.55,-.9);
	\draw (0,-1.5) .. controls ++(90:.3cm) and ++(-135:.2cm) .. (.25,-.9);
	\roundNbox{unshaded}{(-.4,-.4)}{.3}{0}{0}{$p$};
	\roundNbox{unshaded}{(.4,-.4)}{.3}{0}{0}{$p$};
\end{tikzpicture}
=
d_A^{-1/2}
\begin{tikzpicture}[baseline=-.9cm]
	\draw (-.25,-.7) .. controls ++(90:.2cm) and ++(90:.2cm) .. (.25,-.7);
	\draw (-.25,-.7) -- (-.25,-1.1) arc (0:-180:.15cm) -- (-.55,-.4);
	\draw (.25,-.7) -- (.25,-1.1) arc (-180:0:.15cm) -- (.55,-.4);
	\filldraw (-.55,-.9) circle (.05cm);
	\filldraw (.25,-.9) circle (.05cm);
	\draw (-.8,-1.5) .. controls ++(90:.3cm) and ++(-135:.2cm) .. (-.55,-.9);
	\draw (0,-1.5) .. controls ++(90:.3cm) and ++(-135:.2cm) .. (.25,-.9);
\end{tikzpicture}
=
d_A^{-1/2}
\begin{tikzpicture}[baseline=-1.2cm]
	\draw (.25,-.4) -- (.25,-1.1) arc (-180:0:.15cm) -- (.55,-.4);
	\filldraw (0,-1.5) circle (.05cm);
	\filldraw (.25,-.9) circle (.05cm);
	\draw (-.3,-1.8) arc (180:0:.3cm);
	\draw (0,-1.5) .. controls ++(90:.3cm) and ++(-135:.2cm) .. (.25,-.9);
\end{tikzpicture}
= 
\zeta_A\circ m_A.
$$
Finally, we show $\zeta_A$ is involutive, i.e., $\sigma_{\im(p)}\circ \zeta_A = \overline{\zeta_A} \circ \sigma_A$.
First, note that by Lemmas \ref{lem:SymmetricallySelfDualProjector} and \ref{lem:RemoveTopP}, $\sigma_{\im(p)} = p$ on the nose.
Using \eqref{eq:BarStar} and Lemma \ref{lem:RotationRelation} we compute
$$
\overline{\zeta_A} \circ \sigma_A
=
(\zeta_A^*)^\vee \circ \sigma_A
=
d_A^{-1/2}
\begin{tikzpicture}[baseline=-2cm]
	\draw (.65,-.8) -- (.65,-1.4) arc (0:-180:.6cm) -- (-.55,-1.2) arc (180:0:.2cm) -- (-.15,-1.4) arc (-180:0:.2cm) -- (.25,-.8);
	\draw  (-2.05,-2.5) -- (-2.05,-2) arc (180:0:.25cm) arc (-180:0:.25cm) (-1.05,-2) -- (-1.05,-1.45) arc (180:0:.25cm);
	\draw (-1.8,-1.75) -- (-1.8,-1.5);
	\filldraw (-.55,-1.4) circle (.05cm);
	\filldraw (-1.8,-1.75) circle (.05cm);
	\filldraw (-1.8,-1.5) circle (.05cm);
\end{tikzpicture}
=
d_A^{-1/2}
\begin{tikzpicture}[baseline=-1.5cm]
	\draw (0,-.8) -- (0,-1.4) arc (0:-180:.25cm) -- (-.5,-.8);
	\draw (-1,-2) -- (-1,-1.45) arc (180:0:.25cm);
	\draw (-.75,-1.2) -- (-.75,-1);
	\filldraw (-.5,-1.4) circle (.05cm);
	\filldraw (-.75,-1.2) circle (.05cm);
	\filldraw (-.75,-1) circle (.05cm);
\end{tikzpicture}
=
d_A^{-1/2}
\begin{tikzpicture}[baseline=-.1cm]
	\draw (.8,.6) -- (.8,-.3) arc (0:-180:.25cm) (.3,-.3) -- (.3,.6);
	\filldraw (.3,0) circle (.05cm);
	\draw (0,-.6) .. controls ++(90:.3cm) and ++(-135:.3cm) .. (.3,0);
\end{tikzpicture}
=
\zeta_A.
$$
Since $\sigma_{\im(p)}\circ \zeta_A = p\circ \zeta_A = \zeta_A$ by Lemma \ref{lem:RemovePFromAMorphisms}, we are finished.
\end{proof}

\begin{prop}
The involutive algebra isomorphism $\zeta_A \in \cC(A, Q\WStar(A))$ is natural in $A$.
\end{prop}
\begin{proof}
Suppose $\theta \in \cC(A, B)$ is an involutive algebra morphism.
We must prove that $Q\WStar(\theta) \circ \zeta_A = \zeta_B \circ \theta$.
Expanding the definition of $\theta(\alpha)$ for $\alpha \in \cB^{\WStar(A)}_a$, one calculates that
$$
Q\WStar(\theta)
=
d_{A}^{-1/2}d_B^{-1/2}
\sum_{c\in \Irr(\cC)}
\sum_{f\in \ONB(\bfA(a))}
d_c
\begin{tikzpicture}[baseline=-.1cm]
	\draw (0,-1) -- (0,1);
	\draw (0,1.3) .. controls ++(90:.3cm) and ++(-135:.3cm) .. (.3,2);
	\draw (0,-1.3) .. controls ++(270:.3cm) and ++(135:.3cm) .. (.3,-2);
	\draw (.3,2.5) -- (.3,2) arc (-180:0:.25cm) -- (.8,2.5);
	\draw (.3,-2.5) -- (.3,-2) arc (180:0:.25cm) -- (.8,-2.5);
	\filldraw (.3,2) circle (.05cm);
	\filldraw (.3,-2) circle (.05cm);
	\roundNbox{unshaded}{(0,1)}{.3}{0}{0}{$\theta$}
	\roundNbox{unshaded}{(0,-0)}{.3}{0}{0}{$f$}
	\roundNbox{unshaded}{(0,-1)}{.3}{0}{0}{$f^*$}
	\node at (.5,2.4) {\scriptsize{$B$}};
	\node at (1,2.4) {\scriptsize{$\overline{B}$}};
	\node at (-.2,1.5) {\scriptsize{$B$}};
	\node at (-.2,.5) {\scriptsize{$A$}};
	\node at (-.2,-.5) {\scriptsize{$c$}};
	\node at (-.2,-1.5) {\scriptsize{$A$}};
	\node at (.5,-2.4) {\scriptsize{$A$}};
	\node at (1,-2.4) {\scriptsize{$\overline{A}$}};
\end{tikzpicture}
=
d_{A}^{-1/2}d_B^{-1/2}
\begin{tikzpicture}[baseline=-.1cm]
	\draw (0,.3) .. controls ++(90:.3cm) and ++(-135:.3cm) .. (.3,1);
	\draw (0,-.3) .. controls ++(270:.3cm) and ++(135:.3cm) .. (.3,-1);
	\draw (.3,1.5) -- (.3,1) arc (-180:0:.25cm) -- (.8,1.5);
	\draw (.3,-1.5) -- (.3,-1) arc (180:0:.25cm) -- (.8,-1.5);
	\filldraw (.3,1) circle (.05cm);
	\filldraw (.3,-1) circle (.05cm);
	\roundNbox{unshaded}{(0,0)}{.3}{0}{0}{$\theta$}
	\node at (.5,1.4) {\scriptsize{$B$}};
	\node at (1,1.4) {\scriptsize{$\overline{B}$}};
	\node at (-.2,.5) {\scriptsize{$B$}};
	\node at (-.2,-.5) {\scriptsize{$A$}};
	\node at (.5,-1.4) {\scriptsize{$A$}};
	\node at (1,-1.4) {\scriptsize{$\overline{A}$}};
\end{tikzpicture},
$$
where in the second equality, we used that $\sqrt{d_c}\ONB(\bfA(c)) = \sqrt{d_c}\ONB(c,A) = \Isom(c, A)$.
Finally, since $A$ is normalized, we obtain
\begin{equation*}
Q\WStar(\theta)\circ \zeta_A
=
d_B^{-1/2}
\begin{tikzpicture}[baseline=-.4cm]
	\draw (.8,.6) -- (.8,-.3) arc (0:-180:.25cm) (.3,-.3) -- (.3,.6);
	\filldraw (.3,0) circle (.05cm);
	\draw (0,-1.5) -- (0,-.6) .. controls ++(90:.3cm) and ++(-135:.3cm) .. (.3,0);
	\roundNbox{unshaded}{(0,-.9)}{.3}{0}{0}{$\theta$}
	\node at (.5,.5) {\scriptsize{$B$}};
	\node at (1.1,.5) {\scriptsize{$\overline{B}$}};
	\node at (-.2,-.4) {\scriptsize{$B$}};
	\node at (-.2,-1.4) {\scriptsize{$A$}};
\end{tikzpicture}
=
\zeta_B \circ \theta.
\qedhere
\end{equation*}
\end{proof}

\bibliographystyle{amsalpha}
{\footnotesize{
\bibliography{../../../Documents/research/penneys/bibliography}
}}
\end{document}